\documentclass[a4paper,11pt]{article}
\usepackage{amsfonts,amssymb,amsmath,amsthm}

\usepackage[english]{babel}
\usepackage[T1]{fontenc} 
\usepackage[latin1]{inputenc}

\usepackage[]{geometry}
\usepackage{bbm}	
\usepackage{mathrsfs}	
\usepackage{stmaryrd}	

\usepackage{mparhack}
\usepackage[english]{babel}
\newtheoremstyle{theoremstyle}
{\topsep}
{\topsep}
{\itshape}
{-0mm}
{\bfseries}
{.}
{0.5em}
{}

\theoremstyle{theoremstyle}
\newtheorem{theorem}{Theorem}
\newtheorem{lemma}[theorem]{Lemma}
\newtheorem{corollary}[theorem]{Corollary}
\newtheorem{proposition}[theorem]{Proposition}
\newtheorem*{theorem*}{Theorem}
\newtheorem*{lemma*}{Lemma}
\newtheorem*{corollary*}{Corollary}
\newtheorem*{proposition*}{Proposition}
\newtheorem*{defandprop*}{Definition and Proposition}

\theoremstyle{theoremstyle}

\newtheorem*{definition*}{Definition}
\newtheorem*{notation*}{Notation}

\newtheoremstyle{proofstyle}
{0pt}
{5pt}
{}
{-0mm}
{\itshape}
{}
{0.5em}
{}

\theoremstyle{proofstyle}

\newtheoremstyle{remarkstyle}
{0pt}
{5pt}
{}
{}
{\sffamily}
{.}
{0.5em}
{}

\theoremstyle{remarkstyle}
\newtheorem*{remark}{Remark}

\newcommand{\N}{{\mathbb N}}
\newcommand{\Z}{{\mathbb Z}}

\newcommand{\R}{{\mathbb R}}
\renewcommand{\P}{{\mathbb P}}
\newcommand{\E}{{\mathbb E}}
\allowdisplaybreaks[1]

\begin{document}
\title{Time correlations for the parabolic Anderson model\renewcommand{\thefootnote}{\arabic{footnote}}} 

\author{\renewcommand{\thefootnote}{\arabic{footnote}}
{\sc J\"urgen G\"artner}
\footnotemark[1]
\\
\renewcommand{\thefootnote}{\arabic{footnote}}
{\sc Adrian Schnitzler}
\footnotemark[1]
}
\footnotetext[1]{
Institut f\"ur Mathematik, Technische Universit\"at Berlin,
Stra{\ss}e des 17.\ Juni 136, 10623 Berlin, Germany,
{\sl jg@math.tu-berlin.de}, 
{\sl schnitzler@math.tu-berlin.de}
}
\footnotetext[2]{
The work was supported by the DFG International Research Training Group {\it Stochastic Models of Complex Processes
}}

\date{\today}
\maketitle

\begin{abstract}
We derive exact asymptotics of time correlation functions for the parabolic Anderson model with homogeneous initial condition and time-independent tails that decay more slowly than those of a double exponential distribution and have a finite cumulant generating function. We use these results to give precise asymptotics for statistical moments of positive order. Furthermore, we show what the potential peaks that contribute to the intermittency picture look like and how they are distributed in space. We also investigate for how long intermittency peaks remain relevant in terms of ageing properties of the model. 

\vskip 1truecm



\end{abstract}

\section{Introduction}
\subsection{The parabolic Anderson model}
The parabolic Anderson model (PAM) is the heat equation on the lattice with a random potential, given by
\begin{eqnarray}\label{AWPe}
\begin{cases}
\frac{\partial}{\partial t}u(t,x) = \kappa \Delta u(t,x)+ \xi(x) u(t,x),
 \qquad &(t,x)\in (0,\infty)\times\mathbb{Z}^{d},\\ 
 u(0,x)=u_{0}(x), \qquad &x\in \mathbb{Z}^{d},
\end{cases}
\end{eqnarray}
where $\kappa>0$\index{$\kappa$} denotes a diffusion constant, $u_{0}$ a nonnegative function, and $\Delta$ the discrete Laplacian, defined by	
	\[
\Delta f(x) := \sum\limits_{\substack{y\in\mathbb{Z}^{d}:\\|x-y|_1=1}}\left[f(y)-f(x)\right],\qquad x\in \mathbb{Z}^{d},\, f\colon\mathbb{Z}^{d}\rightarrow \mathbb{R}.
\]
Furthermore, $\xi:=\left\{\xi(x),x\in\Z^d\right\}$
 is a random potential.
The solution to \eqref{AWPe} admits the following Feynman-Kac representation (see \cite[Theorem 2.1]{GM90}),
\[u(t,x)=\mathbb{E}_x\exp\Bigg\{\int\limits_0^t \xi\left(X_s\right)\,{\rm d}s\Bigg\}u_0\left(X_t\right),\qquad(t,x)\in [0,\infty)\times\mathbb{Z}^{d},\]
where $X$ is a simple, symmetric, continuous time random walk with generator $\kappa\Delta$ and $\mathbb{P}_x$ ($\mathbb{E}_x$) denotes the corresponding probability measure (expectation) if $X_0=x$ a.s.
  
The solution $u$ depends on two effects. On the one hand, the Laplacian tends to make it flat, whereas the potential causes the occurrence of small regions where almost all mass of the system is located. The latter effect is called {\em intermittency}. It turns out that, the more heavy tailed the potential tails are, the more dominant it becomes. These regions are often referred to as intermittency islands, and the solution $u(t,\cdot)$ develops high peaks on these islands. Commonly the almost sure behaviour of $u$ is referred to as ``quenched'', whereas the behaviour after averaging over the potential $\xi$ is called ``annealed''.

In this paper we will restrict to the case that we have the homogeneous initial condition $u_0\equiv 1$, and that the potential is i.i.d. In this form the PAM was introduced in~\cite{GM90} where existence and uniqueness of the solution have been investigated as well as first order asymptotics for the statistical moments and for the almost sure behaviour of the solution. An overview of the rich literature and recent results on the PAM can be found in~\cite{GK05}. Applications of the PAM are summarised for instance in \cite{M94}.\\
In Section~\ref{sec-results} we formulate our main results.

\subsection{Main results}\label{sec-results}

In this paper we deal with potential tails that decay more slowly than those of a double exponentially (Gumbel) distributed variable $X$, e.g. $\P(X>r)=\exp\{-{\rm e}^{r}\}$ but still have a finite cumulant generating function. Examples that satisfy all conditions that we impose later include the Weibull distribution, i.e., $\P(X>h)=\exp\{-h^\gamma\}$ for $\gamma\in(1,\infty)$. Hence, we are in the first universality class in the classification of~\cite{HKM06}. This class was studied in \cite{GM98}, where some little evidence was gained that the main contribution to the moments of the solution comes from delta-like peaks in the $\xi$-landscape, which are far away from each other. Among other results, they derived the first two terms of the logarithmic asymptotics for the moments of the total mass of the solution. One main result of the present paper, see Section~\ref{sec-momasy}, are the exact asymptotics of these moments. Furthermore, we give a generalisation to more complex functions of the solution evaluated at different times, see Section~\ref{sec-time correlations}.

Another main result, see Section~\ref{sec-intermitt}, describes the height of the intermittency peaks that determine the annealed behaviour. Furthermore, we prove that the complement of the intermittency islands is indeed negligible with respect to the peaks. Since we consider the homogeneous initial condition $u_0\equiv1$, we will investigate the solution in extremely large boxes in which many of these peaks contribute.

Another aspect that we study in this paper are ageing properties of the model. To this end, we compare two notions of ageing, one in terms of time correlations and one in terms of stability of intermittency peaks. In particular, we analyse mixed moments of the solution on two time scales.

Let us formulate more precisely our main assumptions and introduce some notation. By $\left<.\right>$ we denote expectation with respect to $\xi$. The corresponding probability measure is denoted by $\mathbf{P}$. Let $\bar{F}(h):=\mathbf{P}\left(\xi(0)> h \right)$ denote the tail of $\xi(0)$ and $\varphi:=-\log \bar{F}$. Furthermore, let $H(t):=\log\left<{\rm e}^{t\xi(0)}\right>$ be the cumulant generating function of $\xi(0)$. We will make the following assumption on the tails of $\xi$:

\medskip
\underline{Assumption (F)}: 
\begin{enumerate}
 \item If $x\neq y$, then for all $c>0$,
\[\mathbf{P}\left(\frac{\xi(x)+\xi(y)}{2}>h-c\right)=o\left(\bar{F}(h)\right),\qquad h\rightarrow\infty.\]
\item $H(t)<\infty$ \text{  for all } $t\geq0$.
\end{enumerate}
\medskip

Item ii) is equivalent to the existence of moments of the solution of all orders, see \cite{GM90}.
Item i) means that it is much more likely to have one very high peak than to have two quite high peaks. Under Assumption (F) we know that $\lim_{t\rightarrow\infty}H(t)/t=\infty$, i.e., the potential is unbounded to infinity.

To keep the proofs as simple as possible we assume that $\xi$ is bounded from below although analogous results hold true if the potential is unbounded from below. This allows us to assume without loss of generality that
$\text{essinf } \xi=0.$
If $\text{essinf } \xi=c$, we can use the transformation $u\mapsto {\rm e}^{c t}u$ which shifts $\text{essinf } \xi$ to the origin.

\subsubsection{Time correlations}\label{sec-time correlations}
Theorem~\ref{main} provides us with a formula how to compute asymptotically the time correlations for regularly varying functions of the solution $u$. It is also the main proof tool for all further applications.
Spatial correlations for potentials with double exponential or heavier tails  can be found in~\cite{GdH99}, whereas time correlations have not been investigated so far.
Let $ Q_{R}:=[-\lceil R\rceil,\lceil R\rceil]^d\cap{\mathbb Z}^d$ be the $d$-dimensional centered lattice cube of radius $\lceil R\rceil\geq1$ and let
\[\widehat{\xi}_R:=\max\limits_{x\in Q_R\setminus\{0\}}\xi(x).\]
We impose free and zero boundary conditions on the boundary of $Q_R$, denoted by $*={\rm f}$ and $*=0$, respectively. The corresponding Laplacians are denoted by $\Delta_R^*$, that is, for $f\colon Q_R\to\R$,
$$
\Delta_R^{\rm f} f(x)=\sum_{y\in Q_R\colon y\sim x}(f(y)-f(x)),\qquad
\Delta_R^0f(x)=\sum_{y\in \Z^d\colon y\sim x}(f(y)-f(x)),
$$
where for $\Delta_R^0f$ we extend $f$ trivially to $\Z^d$ with the value zero. Zero boundary conditions correspond to $\xi(x)=-\infty$ for $x\notin Q_{R}$. Its law and expectation will be denoted by $\P^{R,0}_x$ and $\E^{R,0}_x$, respectively. The random walk generated by $\Delta_R^{\rm f}$ just remains at its current site at the boundary when the random walk generated by $\Delta$ would jump out of $Q_R$. Its law and expectation will be denoted by $\P^{R,{\rm f}}_x$ and $\E^{R,{\rm f}}_x$, respectively. The corresponding Dirichlet form is given by
\[\left(-\Delta_R^{\rm f} u,u\right)_{ Q_{R}}=\sum\limits_{\substack{\left\{x,y\right\}\in Q_{R}\colon \\|x-y|_1=1}}\big(u(x)-u(y)\big)^2.\]
Let $\lambda^{R,*}_1=\lambda^{R,*}_1(\xi)$ be the principal (i.e., largest) eigenvalue of the Anderson Hamiltonian $\mathscr{H}_R^*:=\kappa\Delta_R^*+\xi$ on $\ell^2\left( Q_{R}\right)$ with free and zero boundary condition, respectively. 

Recall that regularly varying functions are those positive functions $f$ that can be written as $x^\gamma L(x)$, where $\gamma\in\R$ is called the index of variation and $L$ is a slowly varying function called slowly varying part of $f$. 

Let $\mathcal{R},\mathcal{R}_\gamma$ and $\mathcal{R}_+$ be the set of regularly varying functions, regularly varying functions with index of variation $\gamma$, and regularly varying functions with positive index of variation, respectively, with non-decreasing or bounded away from zero and infinity, regularly varying part. Let
\[\mathcal{F}:=\left\{f\in C^1\colon f\in\mathcal{R}_+, f'(x)>0\,\forall x>0, f(0)=0, \lim\limits_{t\rightarrow\infty}f(t)=\infty\right\}\]
and
\[\mathcal{T}:=\left\{f\in C^1\colon f'(x)>0\,\forall x>0, f(0)=0, \lim\limits_{t\rightarrow\infty}f(t)=\infty\right\}.\]

\begin{remark} 
 The fact $f\in\mathcal{R}_+$ already implies that $\lim\limits_{t\rightarrow\infty}f(t)=\infty$, see \cite[Proposition 1.5.1]{BGT87}.
\end{remark}

\begin{theorem}[Time correlations]\label{main}
 Let Assumption (F) be satisfied. Furthermore, let $f_1,\dots,f_p\in\mathcal{F}$ and $t_1,\dots,t_p\in\mathcal{T}$ be given such that for all $a\geq0$,
 \begin{equation}\label{wachs}
  \max\limits_{1\leq j\leq p}{\rm e}^{t_j(t)a}=o\left(\min\limits_{1\leq i\leq p}\left<f_i\left({\rm e}^{t_i(t)\xi(0)}\right)\right>\right),\qquad\text{as }\,t\rightarrow\infty.
 \end{equation}
 Then for every $R\geq1$ and $0<\underline{C}<1<\overline{C}<\infty$ we find that for all $c,t$ large enough,
\begin{eqnarray*} &&\underline{C}\int\limits_{c}^{\infty}\bigg[\frac{{\rm d}}{{\rm d}h}\prod\limits_{i=1}^p f_i\left({\rm e}^{t_i(t)h}\right)\bigg]\mathbf{P}\left(\lambda^{R,0}_{1}(\xi)>h\Big|\widehat{\xi}_R\leq h-c\right)\,{\rm d}h\\
&\leq&\bigg<\prod_{i=1}^p f_i\big(u\left(t_i(t),0\right)\big)\bigg>\\
&\leq&\overline{C}\int\limits_{c}^{\infty}\bigg[\frac{{\rm d}}{{\rm d}h}\prod\limits_{i=1}^p f_i\left({\rm e}^{t_i(t)h}\right)\bigg]\mathbf{P}\left(\lambda^{R,{\rm f}}_{1}(\xi)>h\Big|\widehat{\xi}_R\leq h-c\right)\,{\rm d}h.
\end{eqnarray*}
\end{theorem}
\noindent
Condition \eqref{wachs} determines of what order the functions $t_i$ can be chosen. It is always possible to choose $\max t_i=a\cdot \min t_i$, $a>0$.

Note that Assumption (F) is given in terms of the distribution of the potential, while the asymptotics themselves are expressed in terms of the conditional distribution of the eigenvalues. The asymptotics may be understood as follows. A Fourier expansion in terms of the eigenvalues of $\mathscr{H}_R^*$ yields that
\begin{equation}\label{eigenvExpansion}
u(t,\cdot)\approx {\rm e}^{t\lambda^{R,*}_1(\xi)}\big( e_1^{R,*},{\mathbbm 1}\big) e_1^{R,*}(\cdot),
\end{equation}
where $e_1^{R,*}$ is the positive $\ell^2$-normalised principal eigenfunction. Under Assumption (F), it turns out that the eigenfunction $e_1^{R,*}$ is extremely delta-like peaked. Due to the the requirement $\widehat{\xi}_R\leq h-c$, the peak centre lies in the origin since $\xi(0)$ and $\lambda^{R,*}_1(\xi)$ differ by at most $2d\kappa$.

\subsubsection{Exact moment asymptotics}\label{sec-momasy}

Our first application of Theorem~\ref{main} are exact asymptotics for all moments of positive order. The second order asymptotics for integer moments for a large class of potentials, including the ones that satisfy Assumption (F), can be found in~\cite{GM98}: For any $p\in{\mathbb N}$,
$$
\langle u(t,0)^p\rangle={\rm e}^{H(pt)-2d\kappa pt}\,{\rm e}^{o(t)},\qquad t\to\infty.
$$
We now present much finer asymptotics which are even up to asymptotic equivalence. To the best of our knowledge, this precision has not yet been achieved for the PAM. \\
We need the tails of the principal eigenvalue, conditional on having an extremely high peak at the origin: 
\[{\varphi}_R^*(h):=-\log\mathbf{P}\left(\lambda^{R,*}_1(\xi)>h\,\Big|\,\widehat{\xi}_R\leq h^\alpha\right).\]
Here $\alpha$ is picked according to the following condition which is slightly stronger than Assumption (F).

\medskip

\underline{Assumption (F*)}:
\begin{enumerate}
 \item $\exists\,\alpha<1\colon\quad\bar{F}(h)\cdot\bar{F}(h^\alpha)=o\left(\bar{F}(h+2d\kappa)\right),\qquad h\rightarrow\infty$.
\item $H(t)<\infty$ \text{  for all } $t\geq0$.
\end{enumerate}

\medskip
\noindent
Let $h_t$ be a solution to
\[\sup\limits_{h\in(0,\infty)}\left(th-\varphi(h)\right)=th_t-\varphi\left(h_t\right)=:\psi(t).\]
If $\varphi$ is ultimately convex, then $h_t$ is unique for any large $t$.\\
Now we introduce a condition on the function $\varphi(h)=-\log\P(\xi(0)>h)$. A function $f(t)=o(t)$ is called {\it self-neglecting} if 
\begin{equation}\label{locunifconv}
f\big(t+af(t)\big)\sim f(t),\qquad t\rightarrow\infty,
\end{equation}
locally uniformly in $a\in(0,\infty)$. The convergence in \eqref{locunifconv} is already locally uniform in $a$ if $f$ is continuous (see for instance \cite[Theorem 2.11.1]{BGT87}).\\
Let ${h}^{R,*}_t$ be a solution to
\[\sup\limits_{h\in(0,\infty)}\left(th- \varphi_R^*(h)\right)=t{h}^{R,*}_t-{\varphi}_R^*\left({h}^{R,*}_t\right)=:\psi_R^*(t).\]
If $\varphi$ is ultimately convex, then $h_t^{R,*}$ is unique for any large $t$.

\medskip

\underline{Condition (B)}: The map $t\mapsto \sqrt{\varphi''\left(h_t\right)}$ is self-neglecting.

\medskip
\noindent
Again Condition (B) and Assumption (F*) concern $\xi$ and not $\lambda_1^R$.

\begin{theorem}[Moment asymptotics]\label{exmomw}
Let $\varphi\in C^2$ be ultimately convex, Assumption (F*) and Condition (B) be satisfied and $p\in(0,\infty)$. Then, for any sufficiently large $R$,
\[\left<u(t,0)^p\right>\sim \exp\left\{pt{h}^{R,*}_{pt}-{\varphi}_R^*({h}_{pt}^{R,*})+\log pt +\frac{1}{2}\log \frac{\pi}{({\varphi}_R^*)''({h}_{pt}^{R,*})}\right\},\qquad t\rightarrow\infty.\]
\end{theorem}
\noindent
We see from \eqref{eigenvExpansion} and Theorem~\ref{main} that Theorem~\ref{exmomw} basically follows from an application of the Laplace method.

Note that Weibull tails with parameter $\gamma>1$ satisfy both Condition (F*) and Condition (B). For $\gamma\in(1,3)$, we give an explicit identification of all terms of the asymptotics, see Corollary~\ref{exweib}.

\subsubsection{Relevant potential peaks and intermittency}\label{sec-intermitt}

While originally intermittency was studied by comparing the asymptotics of successive moments of $u$, there have recently been efforts to describe intermittency in a more geometric way by determining time-dependent random sets in $\mathbb{Z}^d$ in which the solution is asymptotically concentrated. These sets are closely related to the support of the leading eigenfunctions of the Anderson Hamiltonian. Clearly, the quenched intermittency picture differs from the annealed one. The height of the quenched intermittency peaks is basically determined by the almost sure growth of the maximal potential peak in a time-dependent box. Its radius depends on the distance that the random walk in the Feynman-Kac representation can make by time $t$. In~\cite{GKM07} the authors describe the geometry of the quenched intermittency peaks for the localised initial condition $u_0=\delta_0$. They find that size and shape of the islands are deterministic, whereas number and location are random. They also give rough bounds on the number and location. They show that under Assumption (F) the quenched intermittency peaks consist of single lattice points.

In contrast, the annealed peaks are significantly higher and occur less frequently.
Their geometry has not been investigated so far. Theorem~\ref{relpo} below determines the height of those potential peaks that contribute to the annealed intermittency peaks, and it proves that the complement contributes a negligible amount. It turns out that the peaks consist of single lattice points as well. 

We will assume from now on that the box $Q_{L_t}$ is chosen so large that the following weak law of large numbers holds true, see \cite[Theorem 1]{BAMR07}:
\begin{equation}\label{WLLN}
 \frac{1}{|Q_{L_t}|}\sum\limits_{x\in Q_{L_t}}u(t,x)\sim\left<u(t,0)\right>,\qquad \mbox{as }t\rightarrow \infty, \mbox{ in probability.}
\end{equation}
To this end, it is sufficient to pick $L(t)$ much larger than $\exp\left\{H(t)\right\}$.
Let 
\[\Upsilon_t^a=\left[h_t-\frac{a}{\sqrt{\varphi''(h_t)}},h_t+\frac{a}{\sqrt{\varphi''(h_t)}}\right],\qquad a>0.\]
In our result it turns out that the set of intermittency peaks may be taken as the set of those sites in which the potential height lies in $\Upsilon_t^a$:

\begin{theorem}[Intermittency]\label{relpo}
 Let Assumption (F*) and Condition (B) be satisfied. Then for every $\varepsilon>0$ there exists $a_\varepsilon$ such that
\[\lim\limits_{t\rightarrow\infty}\mathbf{P}\Bigg(1-\frac{\sum\limits_{x\in Q_{L_t}}u(t,x)\mathbbm{1}_{\xi(x)\in\Upsilon_{t}^a}}{\sum\limits_{x\in Q_{L_t}}u(t,x)}>\varepsilon\Bigg)=
\begin{cases}
1&\mbox{if }a<a_\varepsilon,\\
0&\mbox{if }a>a_\varepsilon.
\end{cases}
\]
\end{theorem}
\noindent
The locations of the peaks form a Bernoulli process, see Corollary~\ref{ber} for details.

\subsubsection{Ageing}\label{sec-ageing}

In this section, we present our results on the dynamic picture of intermittency in the PAM. We will investigate two types of ageing behaviours, {\it correlation ageing} and {\it intermittency ageing}. While the first type gives only rather indirect information about the intermittency peaks, intermittency ageing explicitly describes for how long the intermittency peaks remain relevant. Nevertheless, both approaches give very similar results.

Roughly speaking, a system is ageing if the time it spends in a certain state increases as a function of its current age. An overview of the topic of ageing can be found, for instance, in \cite{BA02}. For the PAM, there have been two approaches. In the case of a (time-dependent) white noise potential $\xi$ as defined in \cite{CM94}, a variant of correlation ageing was investigated in \cite{DD07} and \cite{AD11}. The authors found that there is no ageing. 

In \cite{MOS11} the authors consider a localised initial condition and a time-independent i.i.d.~potential with Pareto-distributed tails. They find that intermittency ageing holds. Their proofs rely on the two cities theorem proved in \cite[Theorem 1.1]{KLMS09} which states that, at any sufficiently late time, all the mass is concentrated in no more than two lattice points, almost surely. 

Let us describe our result on intermittency ageing. As we know from Theorem~\ref{relpo}, there are infinitely many intermittency peaks in our setting, possibly due to the homogeneous initial condition and to the lighter tails, so we have to use a modified definition and different techniques. To define the notion, introduce, for a scale function $s\colon(0,\infty)\to(0,\infty)$,
\[\mathscr{A}_s(t):=\mathbf{P}\Bigg(\Bigg|\frac{\sum\limits_{x\in Q_{L_{t+s(t)}} }u(t,x)\mathbbm{1}_{\xi(x)\in\Upsilon_{t}^a}}{\sum\limits_{x\in Q_{L_{t+s(t)}} }u(t,x)}-\frac{\sum\limits_{x\in Q_{L_{t+s(t)}} }u(t+s(t),x)\mathbbm{1}_{\xi(x)\in\Upsilon_{t}^a}}{\sum\limits_{x\in Q_{L_{t+s(t)}} }u(t+s(t),x)}\Bigg|<\varepsilon\Bigg), \qquad t>0.
\]
Recall that $Q_{L_{t+s(t)}}$ is chosen such that the weak law of large numbers from \eqref{WLLN} holds. We will consider only $a>a_\varepsilon$ as in Theorem~\ref{relpo}. Roughly speaking, $\mathscr{A}_s$ measures whether those potential points that are intermittency peaks at time $t$ are still relevant after time $t+s$.

We define {\em intermittency ageing} by requiring that for any small $\varepsilon>0$ there is $a>0$ and two scale functions $s_1,s_2$ satisfying $\lim_{t\to\infty}s_1(t)=\lim_{t\to\infty}s_2(t)=\infty$ such that
\begin{equation}
 \lim_{t\to\infty}|\mathscr{A}_{s_1}(t) -\mathscr{A}_{s_2}(t)|>0,
\end{equation}
i.e., the two limits of $\mathscr{A}_{s_1}$ and $\mathscr{A}_{s_2}$ both exist and are different.

By the {\it length of intermittency ageing} we understand the class of functions
\[\mathscr{A}:=\left\{s\colon\R\to\R\colon \lim\limits_{t\to\infty}s(t)=\infty,\exists\theta\in (0,\infty)\colon\lim_{t\to\infty}|\mathscr{A}_{s}(t) -\mathscr{A}_{\theta s}(t)|>0\right\}.\] 

\begin{theorem}[Intermittency ageing]\label{age}
Let Assumption (F*) and Condition (B) be satisfied. Then the PAM ages in the sense of intermittency ageing if and only if $\lim_{t\to\infty}H''(t)=0$. In this case $\mathscr{A}\ni 1/\sqrt{H''(t)}=o(t)$.
\end{theorem}
\noindent
For the study of correlation ageing we investigate the following time correlation coefficient
 \[A_f(s,t)=\text{corr}\Big(f\big(u(t,0)\big),f\big(u(t+s(t),0)\big)\Big)=\frac{\text{cov}\Big(f\big(u(t,0)\big),f\big(u(t+s(t),0)\big)\Big)}{\sqrt{\text{var}\Big(f\big(u(t,0)\big)\Big)\text{var}\Big(f\big(u(t+s(t),0)\big)\Big)}}.\]
 Here $f\in C$ is a strictly increasing function with $\lim_{t\rightarrow\infty}f(t)=\infty$.We define {\em correlation ageing} by requiring that there exist two scale functions $s_1,s_2$ satisfying $\lim_{t\to\infty}s_1(t)=\lim_{t\to\infty}s_2(t)=\infty$ such that
\begin{equation}\label{defage}
\lim_{t\to\infty}|A_f(s_1,t) -A_f(s_2,t)|>0.
\end{equation}
\noindent
By the {\it length of correlation ageing} we understand the class of functions
\[\mathcal{A}:=\left\{s\colon\R\to\R\colon \lim\limits_{t\to\infty}s(t)=\infty,\exists\theta\in (0,\infty)\colon\lim_{t\to\infty}|A_f(s,t) -A_f(\theta s,t)|>0\right\}.\] 

\begin{theorem}[Correlation ageing]\label{ageH}
 Let Assumption (F*) and Condition (B) be satisfied and $\varphi\in C^2$ be ultimately convex.
 Then the PAM ages for $f(x)=x^p,p\in\mathbb{R_+}$ in the sense of correlation ageing if and only if $\lim\limits_{t\rightarrow\infty}H''(t)=0$.
In this case $\mathcal{A}\ni 1/\sqrt{H''(t)}=o(t)$.
\end{theorem}
\noindent
Notice that for both definitions ageing happens for lighter tails.
In Theorem~\ref{ageal} we show that Theorem~\ref{ageH} can be extended to more general potentials if we weaken the requirement \eqref{defage}.

\subsection{Overview}
In Section \ref{Time correlations} we prove Theorem~\ref{main} which forms the basis of this paper. In Section \ref{The conditional probability} we show how the conditional probability in Theorem~\ref{main} can be evaluated. After that we will give several applications. In Section \ref{Exact moment asymptotics} we apply Theorem~\ref{main} to prove Theorem~\ref{exmomw} and to derive exact asymptotics for statistical moments and more general functionals of the PAM. In Section \ref{Relevant potential peaks} we prove Theorem~\ref{relpo}. We conclude how the intermittency peaks are distributed in space and give precise estimates on their frequency. In Section \ref{Ageing} we investigate the ageing behaviour of the PAM and prove Theorems~\ref{age}, \ref{ageH} and \ref{ageal}.

\section{Time correlations}\label{Time correlations}

In this section we prove Theorem~\ref{main}.
The strategy of the proof is to show that asymptotically only those realisations of the potential $\xi$ contribute to the expectation $\left<\prod_{i=1}^p f_i\left(u\left(t_i(t),0\right)\right)\right>$, where the highest potential peak $\xi^{(1)}_R$ in the large centered box $Q_R$ is significantly higher than the second one, and where $\xi^{(1)}_R$ is located in the origin. It turns out that for those realisations we can neglect all eigenpairs but the principal one in the spectral representation, and the first eigenfunction becomes delta like. We will see that it is sufficient to consider a large box with time-independent size.
The following universal bounds are always true (see for instance \cite[Theorem 3.1]{GM90} and \cite[Proof of Theorem 2.16]{GM98}). 

\begin{lemma}\label{trivial}  Let $t\geq0$ then for every $R>1$ and $p\in\N$,
\begin{enumerate}
 \item $\lambda_1^{R,*}(\xi)\leq\xi^{(1)}_R\leq\lambda_1^{R,*}(\xi)+2d\kappa$,
 \item ${\rm e}^{H(pt)-2d\kappa pt}\leq\left<u(t,0)^p\right>\leq {\rm e}^{H(pt)}$.
\end{enumerate}
\end{lemma}

\begin{remark}
The lower bound in Lemma~\ref{trivial} ii) can be proven by forcing the random walk $X$ from the Feynman-Kac representation to stay in the origin up to time $t$.
Hence, it remains true if we replace the power function by an arbitrary nonnegative function $f$. Then it reads $\left<f({\rm e}^{t\xi(0)-2d\kappa t})\right>\leq\left<f\big(u(t,0)\big)\right>$.
\end{remark}

Now we show that we can restrict our calculations to an increasing box $ Q_{R_{\widehat{t}}}$ with zero boundary conditions where $R_t:=t\log^2t$ and $\widehat{t}:=\max_{i=1,\dots,p}t_i$. By $\tau_U:=\text{inf}\left\{t>0\colon X_t\in U\right\}$ we denote the first hitting time of a set $U$ by the random walk $X$. For $x\in\mathbb{Z}^d$ we write $\tau_x$ instead of $\tau_{\{x\}}$. Let $u_R$ be the solution to the PAM in $Q_R$ with Dirichlet boundary conditions. Its Feynman-Kac representation is given by
\[u_R\left(t,x\right)=\mathbb{E}_x\exp\Bigg\{\int\limits_0^t\xi(X_s)\,{\rm d}s\Bigg\}\mathbbm{1}_{\tau_{Q_{R}^{\rm c}}\geq t},\qquad(t,x)\in [0,\infty)\times\mathbb{Z}^{d}.\]

 \begin{proposition}\label{cut}
Let $\xi$ be i.i.d., nonnegative and unbounded from above.
If $f_1,\dots,f_p\in\mathcal{F}$ and $t_1,\dots,t_p\in\mathcal{T}$, then
\[\left<\prod_{i=1}^p f_i\big(u\left(t_i(t),0\right)\big)\right>\sim\left<\prod_{i=1}^p f_i\left(u_{R_{\widehat{t}(t)}}\left(t_i(t),0\right)\right)\right>,\qquad t\rightarrow\infty.\]
\end{proposition}

\begin{proof}
Let \[
\widetilde{u}_{R}\left(t,x\right):=u\left(t,x\right)-u_{R}\left(t,x\right)=\mathbb{E}_x\exp\Bigg\{\int\limits_0^t\xi(X_s)\,{\rm d}s\Bigg\}\mathbbm{1}_{\tau_{Q_{R}^{\rm c}}<t},\quad(t,x)\in [0,\infty)\times\mathbb{Z}^{d}.\]
Then for every $\delta>0$ we find that
\begin{eqnarray}
&&\left<\prod_{i=1}^p f_i\left(u\left(t_i,0\right)\right)\right>\nonumber\\
&=&\left<\prod_{i=1}^p f_i\left( u_{R_{\widehat{t}}}\left(t_i,0\right)+\widetilde{u}_{R_{\widehat{t}}}\left(t_i,0\right)\right)\left[\mathbbm{1}_{\forall i\colon\widetilde{u}_{R_{\widehat{t}}}\left(t_i,0\right)\leq\delta u_{R_{\widehat{t}}}\left(t_i,0\right)}+\mathbbm{1}_{\exists i\colon\widetilde{u}_{R_{\widehat{t}}}\left(t_i,0\right)>\delta u_{R_{\widehat{t}}}\left(t_i,0\right)}\right]\right>\nonumber\\
&\leq&
\sum\limits_{T\in\mathcal{P}(p)}\left<\prod_{i\in T}f_i\left( u_{R_{\widehat{t}}}\left(t_i,0\right)\left(1+\delta\right)\right)\prod_{j\in T^{\rm c}}f_j\left(\widetilde{u}_{R_{\widehat{t}}}\left(t_j,0\right)\left(1+\frac{1}{\delta}\right)\right)\right>.\label{tauk}
\end{eqnarray}
Here $\mathcal{P}(p)$ denotes the power set of $\{1,\dots,p\}$ and $T^{\rm c}$ denotes the complement of $T$ within $\{1,\dots,p\}$. 
Since all $f_i$ are regularly varying, it follows that for every $\theta>1$ there exists $\delta=\delta(\theta)$ with $\lim\limits_{\theta\rightarrow 1}\delta(\theta)=0$, and $C_{\theta}$ such that
\begin{equation}\label{delthet}
\max\limits_{i=1,\dots,p}\frac{f_i\big((1+\delta)u\big)}{f_i\left(u\right)}\leq\theta,\qquad u>C_{\theta}.
\end{equation}
Now choose $\theta>1$ arbitrary and fix $\delta>0$ such that \eqref{delthet} is satisfied.\\
Because all $f_i$ are also increasing to infinity we get for large $t$,
\begin{equation}\label{upbou}
\left<\prod_{i=1}^{p}f_i\left( u_{R_{\widehat{t}}}\left(t_i,0\right)\left(1+\delta\right)\right)\right>
\leq\theta\left<\prod_{i=1}^p f_i\left( u_{R_{\widehat{t}}}\left(t_i,0\right)\right)\right>+\prod_{i=1}^{p}f_i\big((1+\delta)C_{\theta}\big).
\end{equation}
Since almost surely
$ u_{R_{\widehat{t}}}\left(t_i(t),0\right)\stackrel{t\rightarrow\infty}{\longrightarrow}\infty$ for all $i$, we can apply Fatou's lemma and see that the asymptotic behaviour of the right hand side of \eqref{upbou} is determined by \[\theta\left<\prod_{i=1}^p f_i\left( u_{R_{\widehat{t}}}\left(t_i,0\right)\right)\right>.\]
By similar arguments we find that there exists $C_{u}$ such that for sufficiently large $t$, 
\begin{eqnarray*} &&\left<\prod_{i\in T}f_i\left( u_{R_{\widehat{t}}}\left(t_i,0\right)\left(1+\delta\right)\right)\prod_{j\in T^{\rm c}}f_j\left(\widetilde{u}_{R_{\widehat{t}}}\left(t_i,0\right)\Big(1+\frac{1}{\delta}\Big)\right)\right>\\
&\leq&C_{u}\left<\prod_{i\in T}f_i\left( u_{R_{\widehat{t}}}\left(t_i,0\right)\right)\prod_{j\in T^{\rm c}}f_j\left(\widetilde{u}_{R_{\widehat{t}}}\left(t_i,0\right)\right)\right>.
\end{eqnarray*}
In a next step we show that
\begin{equation}\label{kompl}
\lim\limits_{t\to\infty}\frac{\left<\prod\limits_{i\in T}f_i\left(u_{R_{\widehat{t}}}\left(t_i,0\right)\right)\prod\limits_{j\in T^{\rm c}}f_j\left(\widetilde{u}_{R_{\widehat{t}}}\left(t_i,0\right)\right)\right>}{\left<\prod\limits_{i\in T}f_i\left(u(t_i,0)\right)\prod\limits_{j\in T^{\rm c}} f_j\left(u(t_j,0)\right)\right>}=0.
\end{equation}
For simplicity we only look at $\lim_{t\to\infty}\left<f_1\big(\widetilde{u}(t,0)\big)\right>/\left<f_1\big(u(t,0)\big)\right>$ which may easily be generalised. Recall that $f_1$ is regularly varying, so it can be written as $f_1(x)=x^\beta L(x)$, for some $\beta>0$ and some slowly varying function $L$. By forcing the random walk in the Feynman-Kac formula to stay in the origin up to time $t$ we find that
\[\left<f_1\big(u(t,0)\big)\right>=\left<u(t,0)^\beta L\big(u(t,0)\big)\right>\geq\left<\exp\{\beta t\xi(0)-2d\kappa \beta t\} L\big(u(t,0)\big)\right>.
\]
Furthermore, together with [GM98, Lemma 2.5] we find that 
\begin{align*}
& \qquad\left<f_1\big(\widetilde{u}(t,0)\big)\right>\\
&\leq\left<\bigg(\exp\{t\xi^{(1)}_{ Q_{R_{t}}}\}\mathbb{P}_0\Big(\tau_{ Q_{R_t}^{{\rm c}}}\leq t\Big)\bigg)^\beta L\big(\widetilde{u}(t,0)\big)\right>\\
&=\left<\bigg(\sum\limits_{x\in Q_{R_{t}}}\exp\{t\xi(x)\}\mathbbm{1}_{\xi(x)=\xi^{(1)}_{ Q_{R_{t}}}}\mathbb{P}_0\Big(\tau_{ Q_{R_t}^{{\rm c}}}\leq t\Big)\bigg)^\beta L\big(\widetilde{u}(t,0)\big)\right>\\
&\leq \max(1,|Q_{R_t}|^{\beta-1})\sum\limits_{x\in Q_{R_{t}}}\left<\bigg(\exp\{t\xi(x)\}2^{d+1}\exp\left\{-R_{t}\log\frac{R_{t}}{\kappa dt}+ nR_{t}\right\}\bigg)^\beta L\big(\widetilde{u}(t,0)\big)\right>\\
&= 2^{(d+1)\beta}\exp\left\{-\beta R_{t}\log\frac{R_{t}}{\kappa dt}+o\left(R_{t}\log\frac{R_{t}}{\kappa d t}\right)\right\}\left<\exp\{\beta t\xi(0)\}L\big(\widetilde{u}(t,0)\big)\right>.\\
\end{align*}
In the third line we use that (due to Jensen's inequality for $\beta>1$) for any real numbers $a_1, \dots ,a_n$, and $\beta>0$,
\begin{align*}
 \bigg(\sum\limits_{k=1}^n a_k\bigg)^\beta \leq \begin{cases}
                                                  n^{\beta-1}\sum\limits_{k=1}^n |a_k| ^\beta,\quad&\text{if }\beta\geq1,\\
\sum\limits_{k=1}^n |a_k| ^\beta,\quad&\text{if }\beta\leq1.
                                                 \end{cases}
\end{align*}

Altogether, this proves \eqref{kompl} since $\widetilde{u}\leq u$ for all $t$ by definition and $L$ is bounded away from zero and infinity or non-decreasing. Therefore, we can conclude that
\[\text{rhs of \eqref{tauk}}-\left<\prod\limits_{i=1}^p f_i\left(u_{R_{\widehat{t}}}\left(t_i,0\right)\right)\right>=o\left(\left<\prod\limits_{i=1}^p f_i\big(u\left(t_i,0\right)\big)\right>\right),\qquad t\rightarrow\infty.\]
Now the claim follows because $\theta$ can be chosen arbitrarily close to 1 and because
\[ 
\left<\prod_{i=1}^p f_i\left(u_{R_{\widehat{t}}}\left(t_i,0\right)\right)\right>\leq\left<\prod_{i=1}^p f_i\left(u\left(t_i,0\right)\right)\right>,\qquad t\geq0,\]
is true by the monotonicity and the nonnegativity of $f_1,\dots,f_p$, and because $u_{R_t}\leq u$ for all $t$.

\end{proof}

The next lemma allows us to consider only those realisations of the potential where the highest potential peak is significantly higher than the second one. Let $\xi^{(1)}_{R_{\widehat{t}}}\geq\xi^{(2)}_{R_{\widehat{t}}}\geq\dots$ be an order statistics of the potential.

\begin{lemma}\label{asympu}
For all $c>0$ and $f_1,\dots,f_p\in\mathcal{F}$, $t_1,\dots,t_p\in\mathcal{T}$ satisfying \eqref{wachs},
\[\bigg<\prod_{i=1}^p f_i\left(u_{R_{\widehat{t}}}^{\rm f}\left(t_i,0\right)\right)\mathbbm{1}_{\xi^{(1)}_{R_{\widehat{t}}}-\xi^{(2)}_{R_{\widehat{t}}}\leq c}\bigg>=o\bigg(\bigg<\prod_{i=1}^p f_i\left(u_{R_{\widehat{t}}}^{\rm f}\left(t_i,0\right)\right)\bigg>\bigg),\qquad t\rightarrow\infty.\]
\end{lemma}

\begin{proof}
Let $x_0,x_1\in Q_{R_{\widehat{t}}}$ be two arbitrarily chosen points. Then
 \begin{eqnarray*}
\bigg<\prod_{i=1}^p f_i\left(u_{R_{\widehat{t}}}^{\rm f}\left(t_i,0\right)\right)\mathbbm{1}_{\xi^{(1)}_{R_{\widehat{t}}}-\xi^{(2)}_{R_{\widehat{t}}}\leq c}\bigg>
&\leq& \bigg<\prod_{i=1}^p f_i\left({\rm e}^{t_i\xi^{(1)}_{R_{\widehat{t}}}}\right)\mathbbm{1}_{\xi^{(1)}_{R_{\widehat{t}}}-\xi^{(2)}_{R_{\widehat{t}}}\leq c}\bigg>\\
&\leq&\sum\limits_{x\in Q_{R_{\widehat{t}}}}\sum\limits_{y\in Q_{R_{\widehat{t}}}\setminus\{x\}}\bigg<\prod_{i=1}^p f_i\left({\rm e}^{t_i\xi(x)}\right)\mathbbm{1}_{\xi(x)>\xi(y)\geq \xi(x)-c}\bigg>\\
&\leq&| Q_{{R_{\widehat{t}}}}|^2\bigg<\prod_{i=1}^p f_i\left({\rm e}^{\frac{t_i}{2}\left(\xi(x_0)+\xi(x_1)+c\right)}\right)\bigg>
\end{eqnarray*}
Since all $f_i$ are regularly varying, they can be written as $f_i(x)=x^{\beta_i} L_i(x)$, where $L_1,\dots, L_p$ are slowly varying functions and $\beta_1,\dots,\beta_i>0$. Therefore, we find that
 \begin{eqnarray*}
 && | Q_{{R_{\widehat{t}}}}|^2\bigg<\prod_{i=1}^p f_i\left({\rm e}^{\frac{t_i}{2}\left(\xi(x_0)+\xi(x_1)+c\right)}\right)\bigg>\\
&=&\bigg<\prod_{i=1}^p{\rm e}^{\beta_i\frac{t_i}{2}\left(\xi(x_0)+\xi(x_1)+2c\right)+2d\log (\widehat{t}\log^2\widehat{t})}{\rm e}^{-\beta_i\frac{t_i}{2}c}L\left({\rm e}^{-\frac{t_i}{2}c}{\rm e}^{\frac{t_i}{2}\left(\xi(x_0)+\xi(x_1)+2c\right)}\right)\bigg>\\
&=& \bigg<\prod_{i=1}^p\underbrace{{\rm e}^{-\beta_i\frac{t_i}{2}c+2d\log (\widehat{t}\log^2\widehat{t})}\frac{L_i\left({\rm e}^{-\frac{t_i}{2}c}{\rm e}^{\frac{t_i}{2}\left(\xi(x_0)+\xi(x_1)+2c\right)}\right)}{L_i\left({\rm e}^{\frac{t_i}{2}\left(\xi(x_0)+\xi(x_1)+2c\right)}\right)}}_{\stackrel{t\to\infty}{\longrightarrow}0} f_i\left({\rm e}^{\frac{t_i}{2}\left(\xi(x_0)+\xi(x_1)+2c\right)}\right)\bigg>.
 \end{eqnarray*}
Assumption (F) states for $c>2d\kappa$:
\[\text {For all }\delta>0\text{ it exists }h_0=h_0(\delta)>0\text{ such that for all } h>h_0\colon\]
\[\mathbf{P}\left(\frac{\xi(x_0)+\xi(x_1)}{2}>h-c\right)\leq\delta\mathbf{P}\left(\xi(0)>h+2d\kappa\right).\]
Furthermore, it follows with Lemma~\ref{trivial} ii) that
\[\bigg<\prod\limits_{i=1}^p f_i\left(u_{R_{\widehat{t}}}^{\rm f}\left(t_i,0\right)\right)\bigg>\geq\int\limits_{0}^\infty \prod\limits_{i=1}^p f_i\left({\rm e}^{t_ih}\right)\mathbf{P}\left(\xi(0)>h+2d\kappa\right)\,{\rm d}h,\]
and therefore, since all $f_i$ are nonnegative and increasing, and because of \eqref{wachs},
\begin{eqnarray*}
\frac{\Big<\prod\limits_{i=1}^p f_i\left({\rm e}^{\frac{t_i}{2}\left(\xi(x_0)+\xi(x_1)+2c\right)}\right)\Big>}{\Big<\prod\limits_{i=1}^p f_i\left(u_{R_{\widehat{t}}}^{\rm f}\left(t_i,0\right)\right)\Big>}
&\leq&\frac{\int\limits_{0}^{h_0} \prod\limits_{i=1}^p f_i\left({\rm e}^{t_ih}\right)\mathbf{P}\left(\frac{\xi(x_0)+\xi(x_1)}{2}>h-c\right)\,{\rm d}h}{\Big<\prod\limits_{i=1}^p f_i\left(u_{R_{\widehat{t}}}^{\rm f}\left(t_i,0\right)\right)\Big>}+\delta\\
&\leq&\frac{h_0\prod\limits_{i=1}^p f_i\left({\rm e}^{t_ih_0}\right)}{\Big<\prod\limits_{i=1}^p f_i\left(u_{R_{\widehat{t}}}^{\rm f}\left(t_i,0\right)\right)\Big>}+\delta\stackrel{t\rightarrow\infty,\delta\rightarrow0}{\longrightarrow}0.
\end{eqnarray*}
\end{proof}
\noindent
Now we prove that the first eigenfunction $e_{1}^{R,*}$ decays at least exponentially fast. To this end we give probabilistic representations for $e_{1}^{R,*}$. 

\begin{lemma}\label{problamb}
 Let $\xi^{(1)}_R-\xi^{(2)}_R=c>2d\kappa$ and $\xi^{(1)}_R=\xi(0)$. Then 
\[e^{R,*}_1(x)=e^{*}_0\mathbb{E}_{x}^{R,*}\exp\bigg\{\int\limits_{0}^{\tau_0}\left(\xi\left(X_{s}\right)-\lambda^{R,*}_{1}(\xi)\right)\,{\rm d}s\bigg\},\qquad x\in\mathbb{Z}^d,\]
where $e^*_0=e^*_0(R)$ is a normalising constant.
\end{lemma}

\begin{proof}
$i)$ $*={\rm f}$. The eigenvalue equation for $\lambda^{R,{\rm f}}_{1}(\xi)$ may be rewritten as
\begin{eqnarray*}
\begin{cases}
&\kappa \Delta e^{R,{\rm f}}_1(x)+(\xi(x)-\lambda^{R,{\rm f}}_{1})e^{R,{\rm f}}_1(x)=0,
 \qquad x\in Q_{R}\backslash\left\{0\right\}, \\ 
 &e^{R,{\rm f}}_1(0)=e^{\rm f}_0.
\end{cases}
\end{eqnarray*}
Since we are working on a finite state space, we know that $\mathbb{E}^{R,{\rm f}}_x\tau_{0}<\infty$ and because $c>2d\kappa$ it also follows that $\xi(x)-\lambda^{R,{\rm f}}_{1}(\xi)<0$ for all $x\in Q_R\setminus\left\{0\right\}$, and hence the Feynman-Kac representation of this boundary problem is given by
\[e^{R,{\rm f}}_1(x)=e^{\rm f}_0\mathbb{E}_{x}^{R,{\rm f}}\exp\Bigg\{\int\limits_{0}^{\tau_0}\left(\xi\left(X_{s}\right)-\lambda^{R,{\rm f}}_{1}(\xi)\right)\,{\rm d}s\Bigg\}.\]
$ii)$ $*=0$. Analogously, the eigenvalue equation for $\lambda^{R,0}_{1}(\xi)$ may be rewritten as
\begin{eqnarray*}
\begin{cases}
&\kappa \Delta e^{R,0}_1(x)+(\xi(x)-\lambda^{R,0}_{1})e^{R,0}_1(x)=0,
 \qquad x\in Q_{R}\backslash\left\{0\right\}, \\ 
 &e^{R,0}_1(0)=e^0_0,\qquad e^0_1(x)=0,\quad x\notin Q_{R}.
\end{cases}
\end{eqnarray*}
Notice that
\[\mathbb{E}_{x}^{R,0}\exp\bigg\{\int\limits_{0}^{\tau_0}\left(\xi\left(X_{s}\right)-\lambda^{R,0}_{1}(\xi)\right)\,{\rm d}s\bigg\}=\mathbb{E}_{x}\exp\bigg\{\int\limits_{0}^{\tau_0}\left(\xi\left(X_{s}\right)-\lambda^{R,0}_{1}(\xi)\right)\,{\rm d}s\bigg\}\mathbbm{1}_{\tau_{0}<\tau_{\left( Q_{R}\right)^{\rm c}}}.\]
Thus, it admits the desired probabilistic representation.
\end{proof}
\noindent
Let
\[\sigma_1:=\inf\left\{t>0\colon X_{t}\neq X_{0}\right\}\]
be the time of the first jump of $X$. The stopping time $\sigma_1$ is exponentially distributed with parameter $2d\kappa$. Now we define recursively
\[\sigma_n:=\inf\left\{t>0\colon X_{t+\sigma_{n-1}}\neq X_{\sigma_{n-1}}\right\}.\]
The sequence $\left\{\sigma_n,n\in\mathbb{N}\right\}$ is i.i.d. by the definition of $X$.

\begin{lemma}\label{eigenexp}
It exists $K=K(c)>0$ such that if $\xi(0)-\widehat{\xi}_R>c>2d\kappa$ then $e_{1}^{R,*}(x)\leq K{\rm e}^{-|x|\log\frac{c}{2d\kappa}}$ for all $x\in Q_R$.
\end{lemma}

\begin{proof}
Using the probabilistic representation of $e_{1}^{R,*}$ from Lemma~\ref{problamb} and because of Lemma~\ref{trivial} we find that
\begin{eqnarray*}
 \frac{e_{1}^{R,*}(x)}{e_0^*}&=&\mathbbm{E}^{R,*}_{x}\exp\Bigg\{\int\limits_{0}^{\tau_{0}}\left(\xi(X_{s})-\lambda^{R,*}_1(\xi)\right)\,{\rm d}s\Bigg\}\\
&\leq&\mathbbm{E}^{R,*}_{x}\exp\left\{\tau_{0}(2d\kappa-c)\right\}\\
&=&\sum\limits_{n=1}^{\infty}\mathbbm{E}^{R,*}_{x}\exp\left\{(2d\kappa-c)\sum\limits_{k=1}^{n}\sigma_{k}\right\}\mathbbm{1}_{\tau_0=\sum\limits_{k=1}^{n}\sigma_{k}}\\
&\leq&\sum\limits_{n=|x|}^{\infty}\left(\frac{2d\kappa}{c}\right)^n=\frac{1}{1-\frac{2d\kappa}{c}}{\rm e}^{-|x|\log\frac{c}{2d\kappa}}.
\end{eqnarray*}
\end{proof}
\noindent
Now we give two facts about regularly varying functions.
\begin{lemma}\label{regular}
 Let $f\in\mathcal{R}_\beta$ be an increasing function and $g\in C$ such that $\lim\limits_{t\rightarrow\infty}g(t)=\infty$. 
\begin{enumerate}
\item  If $\beta>0$, then for any $c>0$ there exists $C(c)$ with $\lim\limits_{c\rightarrow\infty}C(c)=1$ such that \[\sum\limits_{x\in\mathbb{
Z}^d}f\left({\rm e}^{g(t)-c|x|}\right)\sim C(c)f({\rm e}^{g(t)}),\qquad t\rightarrow\infty.\]
\item For every $h\colon\R\to\R$ with $\lim\limits_{t\rightarrow\infty}h(t)=1\colon$ $\lim\limits_{t\rightarrow\infty}\frac{f({\rm e}^{g(t)})}{f\left({\rm e}^{g(t)}h(t)\right)}=1$.
\end{enumerate}
\end{lemma}

\begin{proof}
$i)$ It follows from the uniform convergence theorem for regularly varying functions (see~\cite[Theorem 1.5.2]{BGT87}) that
\begin{eqnarray*}
 \sum\limits_{x\in\mathbb{Z}^d}f\left({\rm e}^{g(t)-c|x|}\right)&=&
f\left({\rm e}^{g(t)}\right)\sum\limits_{x\in\mathbb{Z}^d}\frac{f\left({\rm e}^{g(t)}{\rm e}^{-c|x|}\right)}{f\left({\rm e}^{g(t)}\right)}\\
&\sim&f\left({\rm e}^{g(t)}\right)\sum\limits_{x\in\mathbb{Z}^d}{\rm e}^{-\beta c|x|}
= C(c)f({\rm e}^{g(t)}),\qquad t\rightarrow\infty.
\end{eqnarray*}
$ii)$ follows similarly.
\end{proof}
\noindent
Next we show that we can neglect all eigenvalues but the principal one.
To this end we cut off a time-independent centred inner box $Q_{R}$ from the time-dependent box $Q_{R_t}$ and put free boundary conditions on the inner as well as on the outer side of the boundary of $Q_{R}$. The principal eigenvalue of the outer box will be denoted by $\lambda^{R_t\setminus R,{\rm f}}_1$. Under these circumstances it holds almost surely:
\begin{equation}\label{kirsch}
\max\left(\lambda^{R,{\rm f}}_1,\lambda^{R_t\setminus R,{\rm f}}_1\right)\geq\lambda^{R_t,*}_1\geq \lambda^{R,0}_1.
\end{equation}
The lower bound for $\lambda^{R_t,*}_1$ follows immediately from the Rayleigh-Ritz formula. For the upper bound see for instance \cite[Chapter 5.2]{K07}.
Recall that $\lambda_1^{R,*}>\lambda_2^{R,*}\geq\cdots\geq\lambda^{R,*}_{| Q_{R}|}$ are the eigenvalues of $\mathscr{H}_R^*$, and that the solution $u^*_R$ admits the following spectral representation 
\begin{equation}\label{Spectral}
 u^*_R(t,x)=\sum\limits_{k=1}^{| Q_{R}|} {\rm e}^{\lambda_k^{R,*} t}\left(e_k^{R,*},\mathbbm{1}\right)e_k^{R,*}(x),
\end{equation}
where $e_1^{R,*},e_2^{R,*},\cdots,e_{|Q_R|}^{R,*}$ is a corresponding orthonormal basis of $\ell^2$ eigenfunctions.

\begin{proposition}\label{lowup} 
Let $f_1,\cdots,f_p\in\mathcal{F}$, $t_1,\cdots,t_p\in\mathcal{T}$, \eqref{wachs} be satisfied, and fix $R>0$. Then as $t\to\infty$,
\begin{eqnarray*}
&&\frac{1+o(1)}{| Q_{R}|}\bigg<\prod_{i=1}^p f_i\left({\rm e}^{\lambda^{R,0}_{1}t_i(t)}\right)\Big|\xi^{(1)}_R=\xi(0)\bigg>\\
&\leq & \bigg<\prod_{i=1}^p f_i\left(u_{R_{\widehat{t}(t)}}(t_i(t),0)\right)\bigg>\\
&\leq& \frac{1+o(1)}{| Q_{R_{\widehat{t}(t)}}|}\bigg<\prod_{i=1}^p f_i\left({\rm e}^{\lambda^{R,{\rm f}}_{1}t_i(t)}\right)\Big|\xi^{(1)}_{R_t}=\xi(0)\bigg>.
\end{eqnarray*}
\end{proposition}

\begin{proof}
We will restrict to the case $p=1$ and write $\lambda_{k}^{R_t}$ instead of $\lambda_{k}^{R_t,0}$ to keep the notation simple.\\
Upper bound.\\
It follows by Lemma~\ref{asympu} that for every $c>0$,
\[ \left<f_1\big(u_{R_t}(t,0)\big)\right>=\Big<f_1\big(u_{R_t}(t,0)\big)\mathbbm{1}_{\xi^{(1)}_{R_t}-\xi^{(2)}_{R_t}> c}\Big>+o\big(\left<f_1\big(u_{R_t}(t,0)\big)\right>\big),\qquad t\to\infty.\]
Using the spectral representation \eqref{Spectral} of $u_{R_t}$, we find that there exists $C(c)\in[1,\infty)$ such that as $t$ tends to infinity
\begin{eqnarray*}
 &&\bigg<f_1\big(u_{R_t}(t,0)\big)\mathbbm{1}_{\xi^{(1)}_{R_t}-\xi^{(2)}_{R_t}>c}\bigg>\\
&=&\sum\limits_{x\in Q_{R_t}}\bigg<f_1\bigg(\sum\limits_{k=1}^{| Q_{R_t}|}{\rm e}^{\lambda_{k}^{R_t} t}(e_{k},\mathbbm{1})e_{1}(0)\bigg)\mathbbm{1}_{\xi^{(1)}_{R_t}-\xi^{(2)}_{R_t}> c}\mathbbm{1}_{\xi^{(1)}_{R_t}=\xi(x)}\bigg>\\
&=&\frac{1}{| Q_{R_t}|}
\sum\limits_{x\in Q_{R_t}}\bigg<f_1\bigg({\rm e}^{\lambda_{1}^{R_t} t}\Big((e_{1},\mathbbm{1})e_1(0)+\sum\limits_{k=2}^{| Q_{R_t}|}{\rm e}^{(\lambda_{k}^{R_t}-\lambda_{1}^{R_t})t}(e_{k},\mathbbm{1})e_k(0)\Big)\bigg)\mathbbm{1}_{\xi^{(1)}_{R_t}-\xi^{(2)}_{R_t}> c}\Big|\xi^{(1)}_{R_t}=\xi(x)\bigg>\\
&=&\frac{1+o(1)}{| Q_{R_t}|}C(c)\bigg<f_1\left({\rm e}^{\lambda^{R_t}_{1}t}\right)\mathbbm{1}_{\xi^{(1)}_{R_t}-\xi^{(2)}_{R_t}>c}\Big|\xi^{(1)}_{R_t}=\xi(0)\bigg>\\
 &\leq&\frac{1+o(1)}{| Q_{R_t}|}C(c)\left<f_1\left({\rm e}^{\lambda^{R,{\rm f}}_{1}t}\right)\Big|\xi^{(1)}_{R_t}=\xi(0)\right>.
\end{eqnarray*}
Here we use that $f_1$ is increasing, regularly varying (see Lemma~\ref{regular}), and $f_1(0)=0$. The last inequality follows from the upper bound in \eqref{kirsch}. The expression $(e_{k},\mathbbm{1})e_{1}(0)$ becomes delta-like as $c$ tends to infinity due to Lemma~\ref{eigenexp}. Notice that by the Cauchy-Schwarz inequality and Parseval's identity,
\[\sum\limits_{k=2}^{| Q_{R_t}|}{\rm e}^{(\lambda_{k}^{R_t}-\lambda_{1}^{R_t})t}(e_{k},\mathbbm{1})e_k(0)\leq {\rm e}^{-ct}| Q_{R_t}|^2={\rm e}^{-ct+4d\log(t\log^2 t)}.\]
It follows by Lemma~\ref{regular} i) that $C(c)\stackrel{c\rightarrow\infty}{\longrightarrow}1$.\\ 
\noindent
Lower bound.\\
Similarly as for the upper bound we find asymptotically
\begin{eqnarray*}
 \left<f_1\left(u_{R_t}(t,0)\right)\right>&\geq&\left<f_1\left(u_{R_t}(t)\mathbbm{1}_{\xi^{(1)}_R=\xi(0)}\right)\right>\\
&=&\frac{1}{| Q_{R}|}\left<f_1\Bigg(\sum\limits_{k=1}^{| Q_{R_t}|}{\rm e}^{\lambda_{k}^{R_t} t}(e_{k},\mathbbm{1})e_{k}(0)\Bigg)\Bigg|\xi^{(1)}_R=\xi(0)\right>\\
&\geq&\frac{1+o(1)}{| Q_{R}|}\left<f_1\left({\rm e}^{\lambda^{R,0}_{1}t}\right)\Big|\xi^{(1)}_R=\xi(0)\right>.
\end{eqnarray*}
The last inequality follows by the lower bound in \eqref{kirsch} and arguments similar to those of the upper bound. Notice that $(e_{1},\mathbbm{1})e_{1}(0)\geq(e_{1},e_{1})=1$ if $\xi^{(1)}_{R_t}=\xi(0)$.
\end{proof}
\noindent
Recall that $\widehat{\xi}_R=\max\limits_{x\in Q_R\setminus\{0\}}\xi(x)$.
\begin{lemma}\label{tot}
Let $f_1,\dots,f_p\in\mathcal{F}$, $t_1,\dots,t_p\in\mathcal{T}$ and \eqref{wachs} be satisfied. Then for $R>0$ and $c\in(2d\kappa,\infty)$, as $t\to\infty$,
 \[\bigg<\prod\limits_{i=1}^p f_i\left({\rm e}^{\lambda_1^{R,{\rm f}} t_i}\right)\Big|\xi^{(1)}_{R_{\widehat{t}}}=\xi(0)\bigg>
\sim | Q_{R_{\widehat{t}}}|\int\limits_{c}^{\infty}\bigg[\frac{{\rm d}}{{\rm d}h}\prod\limits_{i=1}^p f_i\left({\rm e}^{t_i(t)h}\right)\bigg]\mathbf{P}\left(\lambda^{R,{\rm f}}_{1}(\xi)>h\Big|\widehat{\xi}_R\leq h-c\right)\,{\rm d}h.\]
\end{lemma}

\begin{proof}
 Let $F_{\lambda_1^{R,{\rm f}}}(h)={\bf P}(\lambda_1^{R,{\rm f}}>h\big|\xi^{(1)}_{R_t}=\xi(0))$. Integration by parts yields
\begin{eqnarray*}
 &&\left<\prod\limits_{i=1}^p f_i\left({\rm e}^{\lambda_1^{R,{\rm f}} t_i}\right)\Big|\xi^{(1)}_{R_{\widehat{t}}}=\xi(0)\right>
=\int\limits_0^\infty\prod\limits_{i=1}^p f_i\left({\rm e}^{h t_i}\right) {\rm d}F_{\lambda_1^{R,{\rm f}}}(h)\\
&=&\prod\limits_{i=1}^p f_i\left(1\right)+\int\limits_{0}^{\infty}\bigg[\frac{{\rm d}}{{\rm d}h}\prod\limits_{i=1}^p f_i\left({\rm e}^{t_ih}\right)\bigg]\mathbf{P}\left(\lambda_1^{R,{\rm f}}>h\Big|\xi^{(1)}_{R_{\widehat{t}}}=\xi(0)\right)\,{\rm d}h.
\end{eqnarray*}
For $c>2d\kappa$, it follows by Lemma~\ref{trivial} and the law of total probability 
\begin{align*}
\mathbf{P}\left(\lambda_1^{R,{\rm f}}>h\Big|\xi^{(1)}_{R_{\widehat{t}}}=\xi(0)\right)
=&| Q_{R_{\widehat{t}}}|\left[\mathbf{P}\left(\lambda_1^{R,{\rm f}}>h\Big|\widehat{\xi}_R\leq h-c\right)\mathbf{P}\left(\widehat{\xi}_R\leq h-c\right)\right.\\
&\left.\qquad+\mathbf{P}\left(\lambda_1^{R,{\rm f}}>h,\xi^{(1)}_{R_{\widehat{t}}}=\xi(0)\Big|\widehat{\xi}_R> h-c\right)\mathbf{P}\left(\widehat{\xi}_R> h-c\right)\right].
\end{align*}
Furthermore, elementary calculus yields
\[\frac{{\rm d}}{{\rm d}h}\prod\limits_{i=1}^p f_i\left({\rm e}^{t_i(t)h}\right)=\sum\limits_{i=1}^{p}t_i(t)t_i'(t)e^{t_i(t)h}f_i'\left(e^{t_i(t)h}\right)\prod_{\substack{j=1\\ j\neq i}}^pf_j\left(e^{t_j(t)h}\right).\]
It remains to show 
\begin{eqnarray}\label{sima}
&&\int\limits_c^\infty {\rm e}^{t_1h}f_1'\left({\rm e}^{t_1h}\right)\prod\limits_{j=2}^p f_j\left({\rm e}^{ t_jh}\right)\mathbf{P}\left(\lambda_1^{R,{\rm f}}>h\Big|\widehat{\xi}_R\leq h-c\right)\mathbf{P}\left(\widehat{\xi}_R\leq h-c\right)\,{\rm d}h\nonumber\\
&\sim&\int\limits_c^\infty {\rm e}^{t_1h}f_1'\left({\rm e}^{t_1h}\right)\prod\limits_{j=2}^p f_j\left({\rm e}^{ t_jh}\right)\mathbf{P}\left(\lambda_1^{R,{\rm f}}>h\Big|\widehat{\xi}_R\leq h-c\right)\,{\rm d}h,\qquad t\rightarrow\infty,
\end{eqnarray}
and
\begin{eqnarray}\label{oa}
 &&\int\limits_0^\infty {\rm e}^{t_1h}f_1'\left({\rm e}^{t_1h}\right)\prod\limits_{j=2}^p f_j\left({\rm e}^{ t_jh}\right)\mathbf{P}\left(\lambda_1^{R,{\rm f}}>h,\xi^{(1)}_{R_{\widehat{t}}}=\xi(0)\Big|\widehat{\xi}_R> h-c\right)\mathbf{P}\left(\widehat{\xi}_R> h-c\right)\,{\rm d}h\nonumber\\
&=&o\Bigg(\int\limits_0^\infty {\rm e}^{t_1h}f_1'\left({\rm e}^{t_1h}\right)\prod\limits_{j=2}^p f_j\left({\rm e}^{ t_jh}\right)\mathbf{P}\left(\lambda_1^{R,{\rm f}}>h\right)\,{\rm d}h\Bigg),\qquad t\rightarrow\infty.
\end{eqnarray}
First we show \eqref{oa}.\\
Assumption (F) implies that for $c>2d\kappa$:
\[\text {For all }\delta>0\text{ it exists }h_0=h_0(\delta)>0\text{ such that for all } h>h_0\colon\]
\[\mathbf{P}\left(\xi(0)>h-c\right)^2\leq\delta\mathbf{P}\left(\xi(0)>h+2d\kappa\right).\]
 It follows with \eqref{wachs} and Lemma~\ref{trivial},
\begin{eqnarray*}
 &&\frac{\text{lhs of }\eqref{oa}}{\int\limits_0^\infty {\rm e}^{t_1h}f_1'\left({\rm e}^{t_1h}\right)\prod\limits_{j=2}^p f_j\left({\rm e}^{ t_jh}\right)\mathbf{P}\left(\lambda_1^{R,{\rm f}}>h\right)\,{\rm d}h}\\
&\leq&| Q_{R}|\Bigg(\frac{\int\limits_0^{h_0} {\rm e}^{t_1h}f_1'\left({\rm e}^{t_1h}\right)\prod\limits_{j=2}^p f_j\left({\rm e}^{ t_jh}\right)\mathbf{P}\left(\xi(0)>h-c\right)^2\,{\rm d}h}{\int\limits_0^\infty {\rm e}^{t_1h}f_1'\left({\rm e}^{t_1h}\right)\prod\limits_{j=2}^p f_j\left({\rm e}^{ t_jh}\right)\mathbf{P}\left(\xi(0)>h+2d\kappa\right)\,{\rm d}h}+\delta\Bigg)\\
&\leq&| Q_{R}|\Bigg(\frac{\prod\limits_{j=1}^p f_j\left({\rm e}^{ t_jh_0}\right){\rm e}^{t_1h_0}}{\prod\limits_{j=2}^p f_j\left(1\right)\left<f_1\left({\rm e}^{t_1\xi}\right)\right>}+\delta\Bigg)\stackrel{t\rightarrow\infty,\delta\rightarrow0}{\longrightarrow}0.
\end{eqnarray*}
Now we show \eqref{sima}.\\
Since $\mathbf{P}(\widehat{\xi}_R\leq h-c)=\left(F(h-c)\right)^{| Q_{R}|-1}\stackrel{h\rightarrow\infty}{\longrightarrow}1$ for fixed $R$ we find for every $\vartheta\in(0,1)$ an $a=a(R,c,\vartheta)$ such that
\begin{eqnarray*}
 &&\int\limits_c^\infty {\rm e}^{t_1h}f_1'\left({\rm e}^{t_1h}\right)\prod\limits_{j=2}^p f_j\left({\rm e}^{ t_jh}\right)\mathbf{P}\left(\lambda_1^{R,{\rm f}}>h\Big|\widehat{\xi}_R\leq h-c\right)\mathbf{P}\left(\widehat{\xi}_R\leq h-c\right)\,{\rm d}h\\
&\geq&\vartheta\int\limits_a^\infty {\rm e}^{t_1h}f_1'\left({\rm e}^{t_1h}\right)\prod\limits_{j=2}^p f_j\left({\rm e}^{ t_jh}\right)\mathbf{P}\left(\lambda_1^{R,{\rm f}}>h\Big|\widehat{\xi}_R\leq h-c\right)\,{\rm d}h.
\end{eqnarray*}
Therefore, it is sufficient to show for every $a\geq0$,
\begin{equation}\label{lang}
\frac{\int\limits_c^a {\rm e}^{t_1h}f_1'\left({\rm e}^{t_1h}\right)\prod\limits_{j=2}^p f_j\left({\rm e}^{ t_jh}\right)\mathbf{P}\left(\lambda_1^{R,{\rm f}}>h\Big|\widehat{\xi}_R\leq h-c\right)\,{\rm d}h}{\int\limits_c^\infty {\rm e}^{t_1h}f_1'\left({\rm e}^{t_1h}\right)\prod\limits_{j=2}^p f_j\left({\rm e}^{ t_jh}\right)\mathbf{P}\left(\lambda_1^{R,{\rm f}}>h\right)\,{\rm d}h}\stackrel{t\to\infty}{\longrightarrow}0.
\end{equation}
Using the bounds from Lemma~\ref{trivial} for $\lambda^{*,R}_1$ and \eqref{wachs}, we find
\begin{equation}
\text{lhs of \eqref{lang} }
\leq\exp\left\{t(a+4d\kappa)+\log{t}-H(t)\right\}
\stackrel{t\rightarrow\infty}{\longrightarrow}0.
\end{equation}
In the last line we use that $\lim_{t\to\infty}H(t)/t=\infty$.
\end{proof}
\noindent
\begin{lemma}\label{tod}
Let $f_1,\dots,f_p\in\mathcal{F}$, $t_1,\dots,t_p\in\mathcal{T}$ and \eqref{wachs} be satisfied. Then for $R>0$ and $c\in(0,\infty)$, as $t\to\infty$,
\[
 \bigg<\prod\limits_{i=1}^p f_i\left({\rm e}^{\lambda_1^{0,R} t_i}\right)\Big|\xi^{(1)}_R=\xi(0)\bigg>\\
\geq \frac{| Q_{R}|}{\big(1+o(1)\big)}\int\limits_{c}^{\infty}\bigg[\frac{{\rm d}}{{\rm d}h}\prod\limits_{i=1}^p f_i\left({\rm e}^{t_ih}\right)\bigg]\mathbf{P}\left(\lambda^{0,R}_{1}(\xi)>h\Big|\widehat{\xi}_R\leq h-c\right)\,{\rm d}h.
\]
\end{lemma}

\begin{proof}
Analog to the proof of Lemma~\ref{tot} we find
 \[\bigg<\prod\limits_{i=1}^p f_i\left({\rm e}^{\lambda_1^{0,R} t_i}\right)\Big|\xi^{(1)}_R=\xi(0)\bigg>\\
=\prod\limits_{i=1}^p f_i\left(1\right)+\int\limits_{c}^{\infty}\bigg[\frac{{\rm d}}{{\rm d}h}\prod\limits_{i=1}^p f_i\left({\rm e}^{t_ih}\right)\bigg]\mathbf{P}\left(\lambda_1^{0,R}>h\Big|\xi^{(1)}_R=\xi(0)\right)\,{\rm d}h.
\]
 Now the claim follows because
\begin{eqnarray*}
 \mathbf{P}\left(\lambda_1^{0,R}>h\Big|\xi^{(1)}_R=\xi(0)\right)
&\geq&\mathbf{P}\left(\lambda_1^{0,R}>h\Big|\widehat{\xi}_R\leq h-c\right)\frac{\mathbf{P}\left(\widehat{\xi}_R\leq h-c\right)}{\mathbf{P}\left(\xi^{(1)}_R=\xi(0)\right)}\\
&\sim& | Q_{R}|\mathbf{P}\left(\lambda_1^{0,R}>h\Big|\widehat{\xi}_R\leq h-c\right),\qquad h\rightarrow\infty.
\end{eqnarray*}
The last asymptotics are proven in the same way as in \eqref{sima}.
\end{proof}
\noindent
\begin{proof}[Proof of Theorem~\ref{main}]
The upper bound is an immediate consequence of Propositions~\ref{cut} and \ref{lowup} and Lemma~\ref{tot}, and the lower bound follows from Propositions~\ref{cut} and \ref{lowup} and Lemma~\ref{tod}.
\end{proof}
\begin{remark}
If we only consider integer moments (i.e. $f_i(x)=x^p, p\in\N$), then the proofs can be simplified and Assumption (F+) below suffices in Theorem~\ref{main} because we can use periodic boundary conditions for the upper bound due to \cite[Lemma 1.4]{GM98}.
\end{remark}

\underline{Assumption (F+)}: 
\begin{enumerate}
\item $\left(\bar{F}(h-c)\right)^2=o\left(\bar{F}(h)\right),\qquad h\rightarrow\infty,\qquad\text{for all }c>0$.
\item $H(t)<\infty$ \text{  for all } $t\geq0$.
\end{enumerate}

\begin{remark}
Theorem~\ref{main} also holds true if we consider the initial condition $u_0=\delta_0$ which implies that in this case for all $f\in\mathcal{R}_+$,
\[\sum\limits_{x\in\mathbb{Z}^d\setminus\left\{0\right\}}\left<f(u(t,x))\right>=o\big(\left<f\big(u(t,0)\big)\right>\big),\qquad t\rightarrow\infty.\]
\end{remark}

\section{The conditional probability}\label{The conditional probability}
In this section we investigate how to calculate $\mathbf{P}\big(\lambda^{*,R}_{1}(\xi)>h\big|\widehat{\xi}_R\leq h-c\big)$. Let 
\[E_R^*(h):=\kappa\sum\limits_{|y|_1=1}\mathbb{E}_{y}^{R,*}\exp\Bigg\{\int\limits_{0}^{\tau_0}\left(\xi\left(X_{s}\right)-h\right)\,{\rm d}s\Bigg\}.\]

\begin{lemma}\label{stochlamb}
Let $h\geq c>2d\kappa$. Then
\[\mathbf{P}\left(\lambda^{R,{\rm f}}_{1}(\xi)>h\Big|\widehat{\xi}_R\leq h-c\right)=\left<\bar{F}\Big(h+2d\kappa-E_R^*(h)\Big)\Big|\widehat{\xi}_R\leq h-c\right>.\]
\end{lemma}

\begin{proof}
Using the probabilistic representation of $e^{R,*}_1$ from Lemma~\ref{problamb}, we find that 
\[\xi(0)=\lambda^{R,*}_{1}(\xi)+2d\kappa-E_R^*\left(\lambda^{R,*}_{1}(\xi)\right),\]
whose right hand side is strictly increasing in $\lambda^{R,{\rm f}}_{1}$. Hence,
\[\lambda^{R,*}_{1}(\xi)>h\quad\Longleftrightarrow\quad\xi(0)>h+2d\kappa-E_R^*(h).\]
Now the statement follows immediately.
\end{proof}
\noindent
\begin{lemma}\label{erlang}
 If $R>0$, $|y|_1=1$, $n\in\mathbb{N}_0$ and $h>\widehat{\xi}_{R}-2d\kappa$, then
\begin{eqnarray*}
&&\sum\limits_{k=0}^{R}n!\binom{n+2k}{n}\mathbb{P}_{ y}\left(\tau_0=\sum\limits_{l=1}^{2k+1}\sigma_l\right)\frac{(2d\kappa)^{2k+1}}{(h+2d\kappa)^{n+2k+1}}\\
 &\leq&\frac{{\rm d}^n}{{\rm d}h^n}\mathbb{E}_{ y}^{R,*}\exp\Bigg\{\int\limits_0^{\tau_0}\left(\xi(X_{s})-h\right)\,{\rm d}s\Bigg\}\\
&\leq&\sum\limits_{k=0}^{\infty}n!\binom{n+2k}{n}\mathbb{P}_{ y}\left(\tau_0=\sum\limits_{l=1}^k\sigma_l\right)\frac{(2d\kappa)^{2k+1}}{(h+2d\kappa-\widehat{\xi}_{R})^{n+2k+1}}.
\end{eqnarray*}
\end{lemma}

\begin{proof}
Let $Y_m:=X_{\sum\limits_{i=1}^m \sigma_i}$
 be the embedded discrete time random walk. We get 
\begin{eqnarray}
&&\frac{{\rm d}^n}{{\rm d}h^n}\mathbb{E}_{ y}\exp\Bigg\{\int\limits_0^{\tau_0}\left(\xi(X_{s})-h\right)\,{\rm d}s\Bigg\}=\mathbb{E}_{ y}\exp\Bigg\{(-\tau_0)^n\int\limits_0^{\tau_0}\left(\xi(X_{s})-h\right)\,{\rm d}s\Bigg\}\nonumber\\
&=&\sum\limits_{k=0}^{\infty}\mathbb{E}_{ y}\Bigg(-\sum\limits_{l=1}^{2k+1}\sigma_l\Bigg)^n\exp\Bigg\{\sum\limits_{l=1}^{2k+1}\sigma_l\left(\xi(Y_{l-1})-h\right)\Bigg\}\mathbbm{1}_{\tau_0=\sum\limits_{l=1}^{2k+1}\sigma_l}.\label{summe}
\end{eqnarray}
Since $\sum\limits_{l=1}^{2k+1}\sigma_l\sim\text{Erlang~}\left(2k+1,2d\kappa\right)$, we find for $z\in \{0,\widehat{\xi}_{R}\}$,
\begin{eqnarray*}
 \mathbb{E}_{ y}{\rm e}^{(z-h)\sum\limits_{l=1}^k\sigma_l}\left(-\sum\limits_{l=1}^k\sigma_l\right)^n&=&\frac{(2d\kappa)^k}{(k-1)!}\int\limits_{0}^{\infty} {\rm e}^{(z-h-2d\kappa) x}x^{n+k-1}\,{\rm d}x\\
&=&\frac{(n+k-1)!}{(k-1)!}\frac{(2d\kappa)^k}{(h+2d\kappa-z)^{n+k}}.
\end{eqnarray*}
Now the claim follows because $0\leq\xi(Y_{l-1})\leq\widehat{\xi}_{R}$ for all $Y_{l-1}\neq 0$ and because jumptimes and jumps are independent.
\end{proof}
\noindent
\begin{lemma}\label{ableit}
Let $\alpha<1$ and fix $R>0$. Then there exists $0<c_1,c_2<\infty$ such that for any large $h$,
\begin{enumerate}
 \item 
        \(\begin{aligned}[t]
&\bar{F}\left(h+2d\kappa-\frac{c_2}{h}\right)\\
\leq&\mathbf{P}\left(\lambda^{R,*}_{1}(\xi)>h,\widehat{\xi}_{R}<h^\alpha\Big|\widehat{\xi}_R\leq h-c\right)\\
\leq&\bar{F}\left(h+2d\kappa-\frac{c_1}{h}\right),
\end{aligned}\)
 \item 
\(\begin{aligned}[t]
&\bar{F}'\left(h+2d\kappa-\frac{c_1}{h}\right)\left(1-\frac{c_1}{(h+2d\kappa)^2}\right)\\
\leq&\frac{{\rm d}}{{\rm d}h}\mathbf{P}\left(\lambda^{R,*}_{1}(\xi)>h,\widehat{\xi}_{R}<h^\alpha\Big|\widehat{\xi}_R\leq h-c\right)\\
\leq& \bar{F}'\left(h+2d\kappa-\frac{c_2}{h}\right)\left(1-\frac{c_2}{(h+2d\kappa)^2}\right),
\end{aligned}\)
\item
\(\begin{aligned}[t]
 &\bar{F}''\left(h+2d\kappa-\frac{c_2}{h}\right)\left(1-\frac{c_1}{(h+2d\kappa)^2}\right)^2+\bar{F}'\left(h+2d\kappa-\frac{ c_1}{h}\right)\frac{c_1}{(h+2d\kappa)^3}\\
\leq&\frac{{\rm d}^2}{{\rm d}h^2}\mathbf{P}\left(\lambda^{R,*}_{1}(\xi)>h,\widehat{\xi}_{R}<h^\alpha\Big|\widehat{\xi}_R\leq h-c\right)\\
\leq&\bar{F}''\left(h+2d\kappa-\frac{c_1}{h}\right)\left(1-\frac{c_2}{(h+2d\kappa)^2}\right)^2+\bar{F}'\left(h+2d\kappa-\frac{ c_2}{h}\right)\frac{c_2}{(h+2d\kappa)^3}.
\end{aligned}\)
\end{enumerate}
\end{lemma}

\begin{proof}
By Lemma~\ref{erlang} we find $0<c_1\leq c_2<\infty$ such that for $|y|_1=1$, any large $h$, and $n\in\N$,
\begin{eqnarray}\label{ableq}
&&\frac{c_1}{(h+2d\kappa)^{n+1}}\nonumber\\
&\leq&\frac{{\rm d}^n}{{\rm d}h^n}\mathbb{E}_y^{R,*}\exp\Bigg\{\int\limits_0^{\tau_0}\left(\xi(X_{s})-h\right)\,{\rm d}s\Bigg\}\\
&\leq&\frac{c_1}{(h+2d\kappa-\widehat{\xi}_{R})^{n+1}}\leq\frac{c_2}{(h+2d\kappa-h^\alpha)^{n+1}}\leq\frac{c_2}{(h+2d\kappa)^{n+1}}.\nonumber
\end{eqnarray}
For large $h$, Lemma~\ref{stochlamb} yields
\[ \mathbf{P}\left(\lambda^{R,*}_{1}(\xi)>h,\widehat{\xi}_{R}<h^\alpha\Big|\widehat{\xi}_R\leq h-c\right)=\left<\bar{F}\left(h+2d-E_R^*(h)\right)\big|\widehat{\xi}_R\leq h^\alpha\right>.\]
The differentiation lemma guarantees that we can differentiate under the integral.
 \begin{enumerate}
  \item follows by substituting the deterministic estimates~\eqref{ableq} in \[\left<\bar{F}\left(h+2d\kappa-E_R^*(h)\right)\big|\widehat{\xi}_R\leq h^\alpha\right>.\]
\item holds because
\begin{eqnarray*}
 &&\frac{{\rm d}}{{\rm d}h}\left<\bar{F}\big(h+2d\kappa-E_R^*(h)\big)\big|\widehat{\xi}_R\leq h^\alpha\right>\\
&=&\left<\bar{F}'\left(h+2d\kappa-E_R^*(h)\right)\left(1-(E_R^*)'(h)\right)\big|\widehat{\xi}_R\leq h^\alpha\right>.
\end{eqnarray*}
Now the claim follows by substituting the deterministic estimates~\eqref{ableq}.
\item holds because
\begin{eqnarray*}
 &&\frac{{\rm d}^2}{{\rm d}h^2}\left<\bar{F}\big(h+2d\kappa-E_R^*(h)\big)\big|\widehat{\xi}_R\leq h^\alpha\right>\\
&=&\left<\bar{F}'\big(h+2d\kappa-E_R^*(h)\big)\left(1-(E_R^*)'(h)^2\right)^2\right.\\
&& +\left.(E_R^*)''(h)\bar{F}'\big(h+2d\kappa-E_R^*(h)\big)\big|\widehat{\xi}_R\leq h^\alpha\right>.
\end{eqnarray*}
Now the claim follows by substituting the deterministic estimates~\eqref{ableq}.
 \end{enumerate}
\end{proof}
\noindent

\section{Exact moment asymptotics}\label{Exact moment asymptotics}
In this section we apply Theorem~\ref{main} to compute exact moment asymptotics via Theorem~\ref{exmomw}. 
To do so we need to impose Assumption (F*) which is a slightly stronger condition than Assumption (F).

\begin{lemma}\label{abal}
 Let Assumption (F*) be satisfied. Then there exists $\alpha\in (0,1)$ such that for any $R>0$ and $c>2d\kappa$,
\begin{equation}\label{f*}
 \frac{\left<\bar{F}\Big(h+2d\kappa-E_R^*(h)\Big)\mathbbm{1}_{\widehat{\xi}_R> h^\alpha}\Big|\widehat{\xi}_R\leq h-c\right>}{\left<\bar{F}\Big(h+2d\kappa-E_R^*(h)\Big)\mathbbm{1}_{\widehat{\xi}_R\leq h^\alpha}\Big|\widehat{\xi}_R\leq h-c\right>}\stackrel{h\to\infty}{\longrightarrow}0.
\end{equation}
\end{lemma}

\begin{proof}
 Asymptotically, we find
\[
 \text{lhs of \eqref{f*} }\leq\frac{\bar{F}(h)\mathbf{P}\left(\widehat{\xi}_R> h^\alpha\right)}{\bar{F}(h+2d\kappa)\mathbf{P}\left(\widehat{\xi}_R\leq h^\alpha\right)}
=\frac{\bar{F}(h)\bar{F}(h^\alpha)}{\bar{F}(h+2d\kappa)}\underbrace{\frac{\left(| Q_{R}|-1\right)}{\left(1-{\rm e}^{-h^{\alpha\gamma}}\right)^{| Q_{R}|-1}}}_{\stackrel{h\rightarrow\infty}{\longrightarrow}1}\stackrel{h\rightarrow\infty}{\longrightarrow}0.\\
\]
\end{proof}
\noindent
In the remainder of this section we will give explicit results for the case that the tails of $\xi(0)$ have a Weibull distribution (i.e. $\bar{F}(h)=\exp\left\{-h^\gamma\right\}$) with parameter $\gamma>1$. One easily checks that Weibull tails satisfy Assumption (F*) for $\alpha>(\gamma-1)/\gamma$.

\begin{lemma}\label{weibg}
Let $\xi\sim$ Weibull ($\gamma$). Then for any $R>0$ and $c>2d\kappa$,
 \begin{eqnarray*}
 &&\mathbf{P}\left(\lambda^{R,{\rm f}}_{1}(\xi)>h\Big|\widehat{\xi}_R\leq h-c\right)\\
&=&\exp\left\{-(h+2d\kappa)^\gamma+2d\kappa^2\gamma(h+2d\kappa)^{\gamma-2}+\mathcal{O}\left((h+2d\kappa)^{\gamma-3}\right) \right\},\qquad h\rightarrow\infty.
\end{eqnarray*}
\end{lemma}

\begin{proof}
It follows from Lemmas~\ref{stochlamb}, ~\ref{erlang} and \ref{abal} and because the lower and the upper bound in Lemma~\ref{erlang} are asymptotically equivalent on a exponential scale that for $\alpha>(\gamma-1)/\gamma$,
\begin{eqnarray*}
 &&\mathbf{P}\left(\lambda^{R,{\rm f}}_{1}(\xi)>h\Big|\widehat{\xi}_R\leq h-c\right)\\
&\sim&\left<\bar{F}\Big(h+2d\kappa-E_R^*(h)\Big)\mathbbm{1}_{\widehat{\xi}_R\leq h^\alpha}\Big|\widehat{\xi}_R\leq h-c\right>\\
&=&\exp\left\{-(h+2d\kappa)^\gamma+2d\kappa^2\gamma(h+2d\kappa)^{\gamma-2}+\mathcal{O}\left((h+2d\kappa)^{\gamma-3}\right) \right\},\qquad h\rightarrow\infty.
\end{eqnarray*}
In the last line we use that
\begin{eqnarray*}
&& \Bigg(h+2d\kappa-\kappa\sum\limits_{|y|_1=1}\mathbb{E}_{y}^{R,{\rm f}}\exp\Bigg\{\int\limits_{0}^{\tau_0}\left(\xi\left(X_{s}\right)-h\right)\,{\rm d}s\Bigg\}\Bigg)^\gamma\\
&=&\left(h+2d\kappa\right)^\gamma-\gamma\left(h+2d\kappa\right)^{\gamma-2}\sum\limits_{|y|_1=1}\frac{h+2d\kappa}{h+2d\kappa-\xi(y)}+\mathcal{O}\big((h+2d\kappa)^{\gamma-4}\big)
\end{eqnarray*}
and
 \[\frac{h+2d\kappa}{h+2d\kappa-\xi(y)}=\sum\limits_{n=0}^{\infty}\left(\frac{\xi(y)}{h+2d\kappa}\right)^n=1+\xi(y)\left((h+2d\kappa)^{-1}+\mathcal{O}\left(h^\alpha(h+2d\kappa)^{\gamma-4}\right)\right).\]
\end{proof}
\noindent
\begin{lemma}\label{probequi}
 Let Assumption (F*) be satisfied. Then there exists $R_0=R_0(\bar{F})$ such that for every $R>R_0$ and $c>2d\kappa$,
\[\mathbf{P}\left(\lambda^{R,{\rm f}}_{1}(\xi)>h\Big|\widehat{\xi}_R\leq h-c\right)\sim\mathbf{P}\left(\lambda^{R,0}_{1}(\xi)>h\Big|\widehat{\xi}_R\leq h-c\right),\qquad h\rightarrow\infty.\]
\end{lemma}

\begin{proof}
 Using the representation from Lemma~\ref{stochlamb}, we see that we only have to look at those paths with $\tau_0\geq\tau_{\left( Q_R\right)^{\rm c}}$ in the left hand side of \eqref{summe} in the proof of Lemma~\ref{erlang}. 
In this case the random walk $X$ must jump at least $2R+1$ times until it reaches the origin for the first time.
Therefore, we have to consider only those summands in the right hand side of \eqref{summe} where $k\geq R$. Together with Lemma~\ref{abal} this yields 
\[\mathbb{E}_{y}\exp\Bigg\{\int\limits_{0}^{\tau_0}\left(\xi\left(X_{s}\right)-h\right)\,{\rm d}s\Bigg\}\mathbbm{1}_{\tau_{0}\geq\tau_{\left( Q_{R}\right)^{\rm c}}}\mathbbm{1}_{\widehat{\xi}_R\leq h-c}=\mathcal{O}\left(h^{-2-2R}\right).\]
\end{proof}
\begin{remark}
 If $\xi\sim\text{Weibull}(\gamma)$, $\gamma>1$, then we can choose $R_0>\frac{\gamma-1}{2}$.
\end{remark}

Now we would like to use a variant of the Laplace method to calculate exact moment asymptotics. This is possible due to the strong Tauberian theorem, Theorem~\ref{strongt}, in the spirit of \cite{FY83}. Recall that $\varphi=-\log\bar{F}$.

\begin{theorem}\label{strongt}
 Let $\varphi\in C^2$ be ultimately convex and Condition (B) be satisfied. Then
\[\exp\left\{H(t)\right\}\sim \exp\left\{th_t-\varphi(h_t)\right\}\sqrt{\frac{2\pi}{\varphi''(h_t)}},\qquad t\rightarrow\infty.\]
\end{theorem}

\begin{remark}
 It follows that
\[H'(t)\sim h_t \text{ and }H''(t)\sim (h_t)'=1/\varphi''(h_t),\qquad t\rightarrow\infty.\]
\end{remark}
\begin{proof}[Proof of Theorem~\ref{exmomw}]
 We find that
all conditions of Theorem~\ref{strongt} are satisfied. Therefore, the claim follows together with Theorem~\ref{main}, where we can take $p=1$, $f_1(x)=x^p$ and $t_1(t)=t$.
\end{proof}
\noindent
In particular we find together with Lemma \ref{weibg} and an asymptotic expansion of ${h}^{R,*}_t$ for Weibull tails:

\begin{corollary}\label{exweib}
Let $\xi(0)\sim\text{Weibull }(\gamma)$, $\gamma>1$ and $p\in(0,\infty)$. Then for $t\rightarrow \infty$,
\begin{eqnarray*}
\left<u(t,0)^p\right>&=&
\exp\bigg\{\left(\gamma-1\right)\left(\frac{p}{\gamma}t\right)^{\frac{\gamma}{\gamma-1}}-2d\kappa p t+2d\kappa^2\gamma\left(\frac{p}{\gamma}t\right)^{\frac{\gamma-2}{\gamma-1}}\\
&&\qquad+\log p t+\frac{1}{2}\log\frac{2\pi}{\gamma(\gamma-1)}\left(\frac{p}{\gamma}t\right)^{-\frac{\gamma-2}{\gamma-1}}+\mathcal{O}\left(t^{\frac{\gamma-3}{\gamma-1}}\right)\bigg\}.
\end{eqnarray*}
\end{corollary}

\section{Relevant potential peaks and intermittency}\label{Relevant potential peaks}
In this section we prove Theorem~\ref{relpo} which tells us what the potential peaks that contribute to the intermittency picture look like and how frequently they occur.
Fix $R>0$ and let
\[\widetilde{\Upsilon}_t^a:=\left[h^{R,*}_t-\frac{a}{\sqrt{( \varphi_R^*)''(h^{R,*}_t)}},h^{R,*}_t+\frac{a}{\sqrt{( \varphi_R^*)''(h^{R,*}_t)}}\right],\qquad a>0,\]
and recall that
\[\Upsilon_t^a=\left[h_t-\frac{a}{\sqrt{\varphi''(h_t)}},h_t+\frac{a}{\sqrt{\varphi''(h_t)}}\right],\qquad a>0.\]
Notice that $1/\varphi''(h_t)=(h_t)'$ and $1/( \varphi_R^*)''(h^{R,*}_t)=(h^{R,*}_t)'$.

\begin{lemma}\label{relen}
 Let Assumption (F) and Condition (B) be satisfied, and $R$ be sufficiently large. Then for every $\varepsilon>0$ there exists $a_\varepsilon$ such that
\[\lim\limits_{t\rightarrow\infty}\frac{\left<u(t,0)\mathbbm{1}_{\lambda_1^{R,*}\in\widetilde{\Upsilon}_{t}^a}\right>}{\left<u(t,0)\right>}\begin{cases}
>1-\varepsilon&\mbox{if }a>a_\varepsilon,\\
<1-\varepsilon&\mbox{if }a<a_\varepsilon.
\end{cases}
\]
\end{lemma}

\begin{proof} 
Recall that $\psi_R^*(t)=th^{R,*}_t-\varphi(h^{R,*}_t)$.
 By Theorems~\ref{main} and \ref{strongt} and with the help of a first order Taylor expansion we see that there exists $\eta_t\in\widetilde{\Upsilon}_t^a$ such that
\begin{eqnarray*}
\left<u(t,0)\mathbbm{1}_{\lambda_1^{R,*}\in\widetilde{\Upsilon}_{t}^a}\right>
&\sim& t\int\limits_{\widetilde{\Upsilon}_t^a} \exp\left\{th- \varphi_R^*(h)\right\}\,{\rm d}h\\
&=& t\frac{\exp\left\{\psi_R^*(t)\right\}}{\sqrt{( \varphi_R^*)''(h^{R,*}_t)}}\int\limits_{-a}^{a}\exp\Bigg\{-\frac{(\varphi_R^*)''(\eta_t)}{2(\varphi_R^*)''(h_t)}u^2\Bigg\}\,{\rm d}u\\
&\sim&t\exp\left\{\psi_R^*(t)\right\}\sqrt{\frac{2\pi}{( \varphi_R^*)''(h^{R,*}_t)}}\big(\Phi(a)-\Phi(-a)\big)\\
&\sim&\left<u(t,0)\right>\big(\Phi(a)-\Phi(-a)\big),\qquad t\rightarrow\infty.
\end{eqnarray*}
Here, $\Phi$ denotes the distribution function of the standard normal distribution. The asymptotic equivalence in the third line is due to Condition (B).
\end{proof}
\noindent
Since $H''(t)\sim1/( \varphi_R^*)''(h^{R,*}_t)\sim1/\varphi''(h_t)$ as $t$ tends to infinity, Lemma~\ref{relen} implies that for all $a>0$, 
\[|\widetilde{\Upsilon}_t^a|\sim|\Upsilon_t^a|\asymp \sqrt{H''(t)},\qquad t\rightarrow\infty.\]
\begin{proof}[Proof of Theorem~\ref{relpo}]
It follows from \cite{GM98} that only those realisations of $\xi$ contribute to the annealed behaviour where $\lambda_1^{R,*}=\xi^{(1)}_R-2d\kappa+o(1)$. Therefore, we find that $h^{R,*}_t=h_t+o(1)$ and hence it follows 
from Lemma~\ref{relen},
\[\lim\limits_{t\rightarrow\infty}\frac{\left<u(t,0)\mathbbm{1}_{\xi(0)\in\Upsilon_{t}^a}\right>}{\left<u(t,0)\right>}\begin{cases}
>1-\varepsilon&\mbox{if }a>a_\varepsilon,\\
<1-\varepsilon&\mbox{if }a<a_\varepsilon.
\end{cases}.\]
Since we have chosen $Q_{L(t)}$ sufficiently large, we can apply the weak LLN and find
\[\frac{\sum\limits_{x\in Q_{L(t)}}u(t,x)\mathbbm{1}_{\xi(x)\in\Upsilon_{t}^a}}{\sum\limits_{x\in Q_{L(t)}}u(t,x)}\stackrel{\mathbb{P}}{\longrightarrow}\frac{\left<u(t,0)\mathbbm{1}_{\xi(0)\in\Upsilon_{t}^a}\right>}{\left<u(t,0)\right>},\]
which completes the proof.
\end{proof}
\noindent
Now we consider the random set
\[\Gamma_t^a:=\left\{x\in Q_{L(t)}\colon\xi(x)\in\Upsilon_{t}^a\right\}.\]
Furthermore, let $\text{Ber}_p$ be a Bernoulli process on the lattice with parameter $p$ and let 
\[i_t^a=\exp\Bigg\{-\varphi\Bigg(h_t+\frac{a}{\sqrt{\varphi''\left(h_t\right)}}\Bigg)\Bigg\}. \]
We find that the spatial picture of the intermittency peaks looks as follows: 

\begin{corollary}\label{ber}
Asymptotically,
\[\Gamma_t^a\sim {\rm Ber}_{i_{t}^a},\qquad t\rightarrow\infty.\]
\end{corollary}

\begin{proof}
 The fact that $\Gamma_t^a$ is a Bernoulli process follows since $\xi$ is i.i.d.\\
The value of the parameter follows because 
\[\mathbf{P}\left(\xi(0)\in\Upsilon_t^a\right)=i^a_t-i^{-a}_t\sim i^a_t,\qquad t\rightarrow\infty. \]
\end{proof}
\noindent
If $\xi(0)\sim$ Weibull($\gamma$), $\gamma>1$, we find with a first order Taylor expansion around $h_t$ that  \[i_{t}^a=\exp\left\{-\left(\frac{t}{\gamma}\right)^{\gamma/\gamma-1}-\frac{at^{-\gamma/2(\gamma-1)}}{\gamma^{1/(\gamma-1)}(\gamma-1)^{1/2}}-\frac{1}{2}a^2\right\}.\]
If we are interested in those potential peaks that are relevant to the $p$-th intermittency peak ($p\in(0,\infty)$), we only have to replace $i^a_t$ by $i^a_{pt}$. It becomes obvious that the $p$-th intermittency peaks at time $t$ correspond to the $q$-th intermittency peaks at time $p/q$.

\section{Ageing}\label{Ageing}
In this section we prove Theorems~\ref{age} and \ref{ageH} to gain a better understanding on how stable the intermittency peaks are.

\subsection{Intermittency Ageing}
We start with the first approach.

\begin{proof}[Proof of Theorem~\ref{age}]
 By Theorem~\ref{relpo} we find 
\[\lim_{t\to\infty}\mathscr{A}_s(t)=\lim\limits_{t\rightarrow\infty}\mathbf{P}\Bigg(\Bigg|1-\frac{\sum\limits_{x\in Q_{L_{t+s(t)}}}u(t+s(t),x)\mathbbm{1}_{\xi(x)\in\Upsilon_{t}^a}}{\sum\limits_{x\in Q_{L_{t+s(t)}}}u(t+s(t),x)}\Bigg|<\varepsilon\Bigg).\]
Recall that $1/\sqrt{ \varphi''(h_t)}=\sqrt{(h_t)'}\sim\sqrt{H''(t)}$. Hence, $\Upsilon_{t}^a=\left[h_t-a\sqrt{(h_t)'},h_t+a\sqrt{(h_t)'}\right]$ and therefore $\Upsilon_t^a\cap\Upsilon_{t+s(t)}^a=\emptyset$ if and only if $h_t+a\sqrt{(h_t)'}-h_{t+s}+a\sqrt{(h_{t+s})'}<0$. For $s=o(t)$ a Taylor expansion yields
\begin{equation}\label{intage}
h_t+a\sqrt{(h_t)'}-h_{t+s}+a\sqrt{(h_{t+s})'}=-s(t)(h_{t})'\big(1+o(1)\big)+2a\sqrt{(h_t)'}\big(1+o(1)\big).
\end{equation}
If $\lim_{t\to\infty}H''(t)>0$, then $\lim_{t\to\infty}(h_t)'>0$ and hence, it follows that the right hand side of \eqref{intage} becomes eventually negative and therefore, $\Upsilon_t^a\cap\Upsilon_{t+s(t)}^a=\emptyset$ for $t$ large and all $s(t)$ tending to infinity.\\
Furthermore, with the help of Fatou's lemma we see that asymptotically
\[\frac{1}{|Q_{L_{t+s(t)}}|}\sum\limits_{x\in Q_{L_{t+s(t)}}}u(t+s(t),x)\mathbbm{1}_{\xi(x)\in\Upsilon_{t}^a}\leq\left<u(t+s(t),0)\mathbbm{1}_{\xi(0)\in\Upsilon_{t}^a}\right>.\]
Now we can conclude that
\[\lim_{t\to\infty}\mathscr{A}_s(t)= \lim\limits_{t\rightarrow\infty}\mathbf{P}\Bigg(\Bigg|1-\frac{\left<u(t+s(t),0)\mathbbm{1}_{\xi(0)\in\Upsilon_{t}^a}\right>}{\left<u(t+s(t),0)\right>}\Bigg|<\varepsilon\Bigg)=0.\]
On the other hand, if $\lim_{t\to\infty}H''(t)=0$ then $\lim_{t\to\infty}(h_t)'=0$ and hence,
if $s(t)=o\left(1/\sqrt{(h_t)'}\right)$ then eventually $\Upsilon_t^a\cap\Upsilon_{t+s(t)}^a=\emptyset$ as above, whereas if $1/\sqrt{(h_t)'}=o\left(s(t)\right)$ then\\ $\lim_{t\to\infty}|\Upsilon_t^a\triangle\Upsilon_{t+s(t)}^a|=0$.\\
In the first case we can proceed as above, while in the second case we find that asymptotically
\begin{align*}
&\frac{1}{|Q_{L_{t+s(t)}}|}\sum\limits_{x\in Q_{L_{t+s(t)}}}u(t+s(t),x)\mathbbm{1}_{\xi(x)\in\Upsilon_{t}^a}\\
\sim&\frac{1}{|Q_{L_{t+s(t)}}|}\sum\limits_{x\in Q_{L_{t+s(t)}}}u(t+s(t),x)\mathbbm{1}_{\xi(x)\in\Upsilon_{t+s(t)}^a}=\left<u(t+s(t),0)\mathbbm{1}_{\xi(0)\in\Upsilon_{t+s(t)}^a}\right>,\quad t\to\infty.
\end{align*}
and hence,
\[\lim_{t\to\infty}\mathscr{A}_s(t)= \lim\limits_{t\rightarrow\infty}\mathbf{P}\Bigg(\Bigg|1-\frac{\left<u(t+s(t),0)\mathbbm{1}_{\xi(0)\in\Upsilon_{t+s(t)}^a}\right>}{\left<u(t+s(t),0)\right>}\Bigg|<\varepsilon\Bigg)=1.\]
\end{proof}

\begin{remark}
 The result remains true if we replace $u$ by $u^p$, $p\in(0,\infty)$.
\end{remark}

\begin{corollary}\label{agew}
 Let $\xi(0)\sim\text{Weibull }(\gamma),\,\gamma>1$. Then the PAM ages in the sense of intermittency ageing if and only if $\gamma>2$.
\end{corollary}

\begin{proof}
It follows from Theorem~\ref{exmomw} that all conditions of Theorem~\ref{age} are satisfied. Now the assumption follows by Theorem~\ref{age} using the asymptotics in Theorem~\ref{exmomw}.
\end{proof}

If we want to know for how long a large peak of height $t_1$ remains relevant, we have to find $t_2=t_2(t_1,a)$ such that
\[h_{t_1}+\frac{a}{\sqrt{\varphi''(h_{t_1})}}=h_{t_2}-\frac{a}{\sqrt{\varphi''(h_{t_2})}}.\]

\subsection{Correlation Ageing}
Now we come to the second approach. 
\begin{proof}[Proof of Theorem~\ref{ageH}]
For simplicity we will only consider $p=1$. Higher powers can be treated analogously. The PAM is always intermittent for stationary potentials, i.e. the second moments are growing much faster than the squares of the first moments. Therefore, it holds 
\[\lim\limits_{t\rightarrow\infty}A_{id}(s,t)=\lim\limits_{t\rightarrow\infty}\frac{\left<u(t,0) u(t+s(t),0)\right>}{\sqrt{\left<u(t,0)^{2}\right>\left<u(t+s(t),0)^{2}\right>}}.\]
It has been shown in \cite[(3.13)]{GM90} that $\left< u(t,0)^pu(s,0)^p\right>=\left< u(t+s,0)^p\right>$ for all $t,s\geq0$ if $u_{0}\equiv1$ for $p=1$. The claim for $p\in\mathbb{N}$ follows by similar techniques. Therefore, together with Theorem~\ref{strongt} we find 
\[\Lambda_p(t):=\log \left<u(t,0)^p\right>=\psi_R^*\left(pt\right)+\frac{1}{2}\log\frac{2\pi}{( \varphi_R^*)''(\widetilde{h}^*_{pt})}+\epsilon(pt)\text{ with }\epsilon(t)=o(1).\]
 Under the given assumptions we find as $t$ tends to infinity,
\begin{eqnarray*}
&&\frac{\left<u(t,0) u(t+s(t),0)\right>}{\sqrt{\left<u(t,0)^{2}\right>\left<u(t+s(t),0)^{2}\right>}}=\frac{\left<u(2t+ s(t),0)\right>}{\sqrt{\left<u(2t,0)\right>\left<u(2(t+s(t)),0)\right>}}\\
&=&\exp\left\{\Lambda_1\big(2t+s(t)\big)-\frac{1}{2}\left[\Lambda_1\left(2t\right)+\Lambda_1\big(2t+2s(t)\big)\right]\right\}\\
&\sim&\exp\underbrace{\left\{\psi_R^*\big(2t+s(t)\big)-\frac{1}{2}\left[\psi_R^*\left(2t\right)+\psi_R^*\big(2t+2s(t)\big)\right]\right\}}_{=:B\big(t,s(t)\big)}\\
&&\times\exp\underbrace{\left\{\epsilon\big(2t+s(t)\big)-\frac{1}{2}\left[\epsilon\left(2t\right)+\epsilon\big(2t+2s(t)\big)\right]\right\}}_{\stackrel{t\rightarrow\infty}{\longrightarrow}0}\\
&&\times\exp\underbrace{\Bigg\{\frac{1}{2}\log\frac{2\pi}{( \varphi_R^*)''\left(\widetilde{h}^*_{2t+s(t)}\right)}-\frac{1}{4}\Bigg[\log\frac{2\pi}{( \varphi_R^*)''\left(\widetilde{h}^*_{2t}\right)}+\log\frac{2\pi}{( \varphi_R^*)''\left(\widetilde{h}^*_{2t+2s(t)}\right)}\Bigg]\Bigg\}}_{=:D\big(t,s(t)\big)}.
\end{eqnarray*}
It is well known that $\Lambda_p\in C^\infty$ and $(\psi_R^*)''>0$.\\
Expanding $B\big(t,s(t)\big)$ into first order Taylor polynomials around $2t+s(t)$ we see that there exist $\eta_1(t)\in\left[2t,2t+s(t)\right]$ and $\eta_2(t)\in\left[2t+s(t),2t+2s(t)\right]$ such that 
\[
 B\big(t,s(t)\big)=-\frac{1}{2}s(t)^{2}\Big((\psi_R^*)''\big(\eta_1(t)\big)+(\psi_R^*)''\big(\eta_2(t)\big)\Big).
\]
Using the estimates from Lemma~\ref{ableit}, we find that 
\[(\psi_R^*)''(t)\sim\psi''(t)\sim H''(t),\qquad t\rightarrow\infty.\]
Case 1: $\lim\limits_{t\rightarrow\infty}H''(t)>0$.\\
In this case it follows that $\lim_{t\rightarrow\infty}B\big(t,s(t)\big)=-\infty$ which implies $\lim_{t\to\infty}A_{{\rm id}}(s(t),t)=0$, for all $s$.
\newline
\newline
Case 2: $\lim\limits_{t\rightarrow\infty}H''(t)=0$.\\
Remember that under Assumption $F$ we have $t=o(H(t))$ and hence $t^{-1}=o(H''(t))$.\\ 
In this case we find two constants $0<C_1<C_2<\infty$ such that
\[-C_1H''(2t)s(t)^{2}\leq B\big(t,s(t)\big)\leq -C_2H''\big(2t+s(t)\big)s(t)^{2}.\]
Consequently, if we choose $s$ such that $\lim_{t\rightarrow\infty}H''\big(2t+2s(t)\big)s(t)^{2}=\infty$ it follows that $A_{id}(s)=0$, whereas if we choose $s$ such that $\lim_{t\rightarrow\infty}H''(2t)s(t)^{2}=0$ it follows that $\lim_{t\to\infty}A_{{\rm id}}(s,t)=1$.\\
We observe that $1/\sqrt{H''(t)}\in\mathcal{A}$ which implies that both regimes can occur for functions $s$ of order $o(t)$.
Because of Condition (B), we find that $1/( \varphi_R^*)''\left(\widetilde{h}^*_{t}\right)=o(t)$ and hence $D\big(t,s(t)\big)$ tends to zero as $t$ tends to infinity if $s=o(t)$.
Theorem~\ref{main} is applicable for $s=o(t)$ and implies that \[\left<u(t,0)^pu\big(t+s(t)\big)^p\right>\sim\left<u\big(p\big(2t+s(t)\big),0\big),0\big)\right>,\qquad t\rightarrow\infty,\]
 for $p\in(0,\infty)$. Hence, we can generalise the result to positive real exponents which completes the proof.
\end{proof}
\noindent
Notice that the PAM ages if and only if the length of the intervals $\Upsilon_t^a$ tends to zero as $t$ tends to infinity. In this case we find that $1/|\Upsilon_t^a|\in\mathcal{A}$, for all $a>0$.

\begin{corollary}\label{ageW}
 Let $\xi(0)\sim\text{Weibull }(\gamma),\,\gamma>1$. Then the PAM ages for $f=x^p,p\in\mathbb{R_+}$ in the sense of correlation ageing if and only if $\gamma>2$.
\end{corollary}

\begin{proof}
Analog as in Corollary~\ref{agew}.
\end{proof}
\noindent
We see that for Weibull tails the order of the length of ageing is increasing in $\gamma$.

The main obstacle in proving Theorem~\ref{ageH} is that we have to show that $\frac{{\rm d}^2}{{\rm d}t^2}\log\left<u(t,0)\right>$ is not fluctuating too much. We have proven this under Assumptoin (F*). For more general potentials we are still able to prove correlation ageing if we replace \eqref{defage} by only requiring that
\[\liminf_{t\to\infty}|A_f(s_1,t) -A_f(s_2,t)|>0,\]
and by modifying the definition of $\mathcal{A}$ accordingly. We call this {\it weak correlation ageing}. It has been proven in \cite{HKM06} that under some mild regularity assumptions on $\xi$ there exists a non-decreasing and regularly varying scale function $\alpha\colon(0,\infty)\to(0,\infty)$ with $\alpha(t)=o(t)$ and a constant $\chi\in\R$ such that
\begin{equation}\label{univers}
 \log \left<u(t,0)\right>=H\left(\frac{t}{\alpha(t)^d}\right)\alpha(t)^d+\frac{t}{\alpha(t)^2}\left(\chi+o(1)\right).
\end{equation}
Let $H_\alpha(t):=H\left(\frac{t}{\alpha(t)^d}\right)\alpha(t)^d+\frac{t}{\alpha(t)^2}\chi$ and $h(t):=\log \left<u(t,0)\right>-H_\alpha(t)$.
Then we find:

\begin{theorem}\label{ageal}
 Let Assumptions (H) and (K) from \cite{HKM06} be satisfied. 
If $\lim\limits_{t\to\infty}\alpha''(t)$ exists and $\lim\limits_{t\to\infty}H''(t)=0$, then the PAM is weakly correlation ageing for $f=id$.
\end{theorem}

\begin{proof}
 Since $\alpha(t)=o(t)$ and  $\lim_{t\to\infty}\alpha''(t)$ exists this limit must be zero. Therefore, and because $\lim_{t\to\infty}H''(t)=0$ it follows that $\lim_{t\to\infty}H_\alpha''(t)=0$, as well. It follows from \eqref{univers} that $h(t)=o\left(H_\alpha(t)\right)$. Altogether, this and the convexity of $ \log \left<u(t,0)\right>$ imply that for every sequence of intervals $I(t):=\left[l(t),2l(t)\right]$ with $\lim_{t\to\infty}l(t)=\infty$,
\[\frac{\lambda\bigg(s\in I(t)\colon h''(s)\notin \left[-H_\alpha''(t),H_\alpha''(t)\right]\bigg)}{\lambda\bigg(s\in I(t)\colon h''(s)\in \left[-H_\alpha''(t),H_\alpha''(t)\right]\bigg)}\stackrel{t\to\infty}{\longrightarrow}0,\]
 where $\lambda$ denotes the Lebesgue measure. Now the claim follows with a Taylor expansion as in the proof of Theorem~\ref{ageH}.
\end{proof}
\noindent
For the first and the second universality class in the classification of \cite{HKM06} we find that $\alpha$ is constant, and hence, the requirement that $\lim\limits_{t\to\infty}\alpha''(t)$ exists is fulfilled for all distributions in these classes. Here it holds that 
 $1/\sqrt{H''(t)}\in\mathcal{A}$.

\subsection*{Acknowledgement}
We would like to thank Wolfgang K\"onig for various helpful comments on the first draft of this paper. Furthermore, we would like to thank an attentive referee.

\end{document}